\documentclass[12pt, reqno]{amsart}
\usepackage{amsmath, amsthm, amscd, amsfonts, amssymb, mathtools, color, hyperref}
\usepackage{mwe}
\usepackage{forest}
\useforestlibrary{edges}
\forestset{qtree/.style={for tree={parent anchor=south, 
			child anchor=north,align=center,inner sep=0pt}}}
\usepackage{flexisym}

\makeatletter
\def\thm@space@setup{\thm@preskip=1.25pt
\thm@postskip=0pt}
\makeatother
\newtheoremstyle{newstyle}      
{} %Aboveskip 
{} %Below skip
{\mdseries} %Body font e.g.\mdseries,\bfseries,\scshape,\itshape
{} %Indent
{\bfseries} %Head font e.g.\bfseries,\scshape,\itshape
{.} %Punctuation afer theorem header
{ } %Space after theorem header
{} %Heading

\newtheorem{theorem}{Theorem}[section]
\newtheorem{lemma}[theorem]{Lemma}

\newtheorem{corollary}[theorem]{Corollary}
\newtheorem{observation}[theorem]{Observation}

\theoremstyle{definition}
\newtheorem{definition}[theorem]{Definition}

\newtheorem{conjecture}[theorem]{Conjecture}

\theoremstyle{remark}
\newtheorem{remark}[theorem]{Remark}
\numberwithin{equation}{section}
\newcommand{\reals}{\mathbb{R}}



\begin{document}
\setcounter{page}{1}

\title[Semipositivity of matrices - linear preservers]
{On linear preservers of semipositive matrices}

\author[S. Jayaraman]{Sachindranath Jayaraman}
\address{School of Mathematics\\ 
Indian Institute of Science Education and Research Thiruvananthapuram\\ 
Maruthamala P.O., Vithura, Thiruvananthapuram -- 695 551, Kerala, India.}

\email{sachindranathj@iisertvm.ac.in, sachindranathj@gmail.com}

\author[V. N. Mer]{Vatsalkmar N. Mer}
\address{SQC \& OR Unit\\ 
Indian Statistical Institute\\ 
7, S. J. S. Sansanwal Marg,, New Delhi - 110 016, India.}

\email{vatsal.n15@iisertvm.ac.in, vnm232657@gmail.com}



\begin{abstract}
Given proper cones $K_1$ and $K_2$ in $\reals^n$ and $\reals^m$, respectively, an 
$m \times n$ matrix $A$ with real entries is said to be semipositive if there exists 
a $x \in K_1^{\circ}$ such that $Ax \in K_2^{\circ}$, where $K^{\circ}$ denotes the 
interior of a proper cone $K$. This set is denoted by $S(K_1,K_2)$. We resolve a recent 
conjecture on the structure of into linear preservers of $S(\reals^n_+,\reals^m_+)$. We 
also determine linear preservers of the set $S(K_1,K_2)$ for arbitrary proper cones $K_1$ 
and $K_2$. Preservers of the subclass of those elements of $S(K_1,K_2)$ with a 
$(K_2,K_1)$-nonnegative left inverse as well as connections between strong linear 
preservers of $S(K_1,K_2)$ with other linear preserver problems are considered.
\end{abstract}

\maketitle

\section{Introduction}

We work throughout over the field $\reals$ of real numbers.
Let $M_{m,n}$ denote the set of all $m \times n$ matrices over $\reals$. 
When $m = n$, this set will be denoted by $M_{m}$ or $M_{n}$. 
A matrix $A \in M_{m,n}$ is said to be semipositive if there exists a 
$x > 0$ such that $Ax > 0$, where the inequalities are understood componentwise. 
$A$ is said to be minimally semipositive if it is semipositive and no proper 
$m \times p$ submatrix of $A$ is semipositive for $p < n$. $A$ is said to be 
redundantly semipositive if it is semipositive but not minimally semipositive. 
It is known that an $m \times n$ matrix $A$ is minimally semipositive if and only 
if $A$ is semipositive and $A$ has a nonnegative left inverse. Semipositivity 
characterizes invertible M-matrices within the class of Z-matrices (See Chapter 6, 
\cite{bp-book}). For recent results on semipositive matrices, their structure and preservers, one may refer to \cite{csv-1,prs-1,dgjjt,tsat-2} and the references 
cited therein.

For a field $\mathbb{F}$ and the set $M_{m,n}(\mathbb{F})$ of $m \times n$ matrices over 
$\mathbb{F}$, a linear preserver $L$ is a linear map $L : M_{m,n}(\mathbb{F}) 
\longrightarrow M_{m,n}(\mathbb{F})$ that preserves a certain property or a relation. 
There are two types of preserver problems. Given a subset $\mathcal{S}$ of 
$M_{m,n}(\mathbb{F})$, characterize linear maps $L$ on 
$M_{m,n}(\mathbb{F})$ such that $(i) \ L(\mathcal{S}) \subset \mathcal{S} \ 
\text{and} \ (ii) \ L(\mathcal{S}) = \mathcal{S}$. The first one is called an 
\textit{into} preserver and the latter an \textit{onto/strong} preserver. There is 
rich history on this topic within linear algebra as well as other areas of mathematics. 

Recall that an $m \times n$ matrix $A$ is row positive if $A$ is a nonnegative 
matrix with at least one nonzero entry in each row. A square matrix $B$ is said to 
be a monomial matrix if in addition to being a nonnegative matrix, each row and column 
of $B$ contains exactly one nonzero entry. The starting point and motivation for 
this work comes from the following results (Theorems 2.4, 2.11 and Corollary 2.7) 
due to Dorsey et al \cite{dgjjt}. 

\begin{enumerate}
\item (Theorem 2.4, \cite{dgjjt}) Let $L(A) = XAY$ for some $X \in M_m$ and $Y \in M_n$. 
Then $L$ is an into preserver of semipositivity if and only if $X$ is row positive 
and $Y$ is inverse nonnegative, or $-X$ is row positive and $-Y$ is inverse nonnegative. 
$L$ is an onto preserver of semipositivity if and only if $X$ and $Y$ are monomial, 
or $-X$ and $-Y$ are monomial (Corollary 2.7, \cite{dgjjt}).
\item (Theorem 2.11, \cite{dgjjt}) Let $L(A) = XAY$ for some $X \in M_m$ and 
$Y \in M_n$. Then $L$ is an into preserver of minimal semipositivity if and only 
if $X$ is monomial and $Y$ is inverse nonnegative, or $-X$ is monomial and $-Y$ is 
inverse nonnegative. $L$ is an onto preserver of minimal semipositivity if and only 
if $X$ and $Y$ are monomial, or $-X$ and $-Y$ are monomial.
\end{enumerate}

Our primary aim in this manuscript is to resolve the following conjecture due to 
Dorsey et al \cite{dgjjt}:

\begin{conjecture}\label{conj-1}
Let $L: M_{m,n} \rightarrow M_{m,n}$ be an invertible linear map. 
If $L$ is an into preserver of $S(\reals^n_+,\reals^m_+)$, then $L(A) = XAY$ for all 
$A \in M_{m,n}$, where $X$ is row positive and $Y$ is inverse nonnegative.
\end{conjecture}

Dorsey et al had pointed out through an example that invertibility of the map $L$ is 
crucial in the above conjecture (see the Example in Section $4$ of \cite{dgjjt}). 
They had also added that there was computational evidence that the conjecture is true 
in the $2 \times 2$ case, but had no proof nor a counterexample. 

The purpose of this manuscript is twofold. The first one concerns resolving Conjecture \ref{conj-1}. 
The second part concerns linear preservers of the set $S(K_1,K_2)$ (see 
the next section for definitions and notations). We begin by recalling preliminaries 
about convex sets, nonnegative and semipositive matrices. Other necessary results on 
convex cones and positive operators are presented in a later section when we 
discuss linear preservers of $S(K_1,K_2)$. Our first result says that if an invertible linear map $L$ 
preserves the set of semipositive matrices and also maps every rank one semipositive matrix to a rank 
one (semipositive) matrix, then $L$ is a rank one preserver (Theorem \ref{theorem-suff-cond}) and 
consequently, $L(A) = XAY$ for invertible matrices 
$X$ and $Y$. This result provides a hint for resolving Conjecture \ref{conj-1}. Our main 
result says that an invertible linear map $L$ on $M_{m,n}$ that preserves semipositive matrices 
is of the form $L(A) = XAY$ for all $A \in M_{m,n}$ for some invertible row positive matrix $X \in M_m$ 
and inverse nonnegative matrix $Y \in M_n$ if and only if 
every rank one semipositive matrix of a special form gets mapped to a rank one matrix 
(see Theorem \ref{conclusion}). The structure of invertible maps $L$ on $M_2$ that 
preserve semipositivity (Theorems \ref{2x2-main-thm}) is discussed following 
Theorem \ref{theorem-suff-cond}. A similar argument also works for maps on $M_{m,2}$. 
The proof is constructive and we deduce that $L(A) = XAY$ for some row positive matrix 
$X$ and inverse nonnegative $Y$. The general case is taken up next. The proof follows 
ideas that were verified for an invertible linear map $L$ on $M_3$ (and more generally 
on $M_{m,3}, m \geq 3$) that preserves semipositive matrices. Since the calculations 
are involved and lengthy, the details of the $3 \times 3$ case are not included in 
the main part of the manuscript. We include these as an appendix. The second part of 
the manuscript concerns linear preservers of the set $S(K_1,K_2)$. 
This section is subdivided into further subsections as follows: 
(1) Useful results (2) Linear preservers of $S(K_1,K_2)$ 
(3) Preservers of $MS(K_1,K_2)$ (see the next section for the definition) (4) General 
strong preservers of $S(K_1,K_2)$ and (5) Left semipositivity and their preservers. Interesting 
connections between onto preservers of $S(K_1,K_2)$ and into preservers of nonnegativity 
are also brought out. To the best of our knowledge this appears new and completely 
settles this problem both over all proper cones, including the nonnegative orthants.

\begin{remark}
We shall use the same notation for a linear map $T: \reals^n \rightarrow \reals^m$ and 
its matrix representation, which we assume throughout to be with respect to the standard basis of 
$\reals^n$ and $\reals^m$, respectively.
\end{remark}

\section{Preliminary Results}\hspace{\fill}\\

We present the preliminary results in this section. 
 
\subsection{Convex cones, nonnegative and semipositive matrices}\hspace{\fill}\\
 
Let us recall that a subset $K$ of $\reals^n$ is called a convex cone if 
$K + K \subseteq K$ and $\alpha K \subseteq K$ for all $\alpha \geq 0$. $K$ is said to 
be proper if it is topologically closed, pointed ($K \cap -K = \{0\}$) and has nonempty 
interior $K^{\circ}$. $K$ is said to be polyhedral if $K = X(\reals^m_+)$ for some 
$n \times m$ matrix $X$ and simplicial when $X$ is invertible. The dual, 
$K^{\ast}$, is defined as 
$K^{\ast} = \{y \in \reals^n : \langle y, x \rangle \geq 0 \forall x \in K\}$, 
where $\langle .,. \rangle$ denotes the usual Euclidean inner product on $\reals^n$. 
When $K$ is a convex cone in $\reals^n$ such that $K = K^{\ast}$, we say that $K$ is 
a self-dual cone in $\reals^n$. The most well-known example of a proper (convex) 
self-dual cone is the nonnegative orthant in $\reals^n$: \\ 
$K = \reals^n_+ = \{x = (x_1, \ldots, x_n)^t \in \reals^n : x_i \geq 0 \ 
\forall \ 1 \leq i \leq n\}$. We assume that all cones in this manuscript are proper 
cones.

The following definitions, notations and basic results are standard (see, for instance 
\cite{barker-cones}). The only exception is Definition \ref{defn-2}, which is a natural 
generalization to proper cones the notion of minimally semipositive matrices (see 
\cite{jks} for the definition). 

\begin{definition}\label{defn-1}
For proper cones $K_1$ and $K_2$ in $\reals^n$ and $\reals^m$, respectively, we have 
the following notions. $A \in M_{m,n}$ is 
\begin{enumerate}
\item $(K_1,K_2)$-nonnegative if $A (K_1) \subseteq K_2$.
\item $(K_1,K_2)$-semipositive if there exists a $x \in K_1^\circ$ such that 
$Ax \in K_2^\circ$.
\end{enumerate}
\end{definition}

We denote the set of all matrices that are $(K_1,K_2)$-nonnegative by 
$\pi(K_1,K_2)$. When $K_1 = K_2 = K$, this will be denoted by $\pi(K)$. Let us also 
denote the set of all matrices that are $(K_1,K_2)$-semipositive  by $S(K_1,K_2)$. 
When $K_1 = K_2 = K$, this will be denoted by $S(K)$.

\begin{definition}\label{defn-2}
Let $A \in M_{m, n}$ be $(K_1, K_2)$-semipositive. We say $A$ is 
$(K_1, K_2)$-minimally semipositive  if $A$ has a $(K_2, K_1)$-nonnegative left inverse.
\end{definition}

The set of all $(K_1, K_2 )$-minimally semipositive matrices will be denoted by 
$MS(K_1, K_2)$.
 
\begin{definition}\label{defn-3}
A square matrix $Y$ is $K$-inverse nonnegative if $Y$ is invertible with $Y^{-1}$ being 
$K$-nonnegative.
\end{definition} 

\section{Main Results}

The main results of this manuscript are presented in this section. As stated previously, 
our aim in this manuscript is to resolve Conjecture \ref{conj-1}. The primary motivation 
for this problem comes from Theorems 2.4, 2.11 and 3.5 due to Dorsey et al \cite{dgjjt}. 
We assume that $m \geq n$ and $n \geq 2$. 

\subsection{The nonnegative orthants case} \hspace{\fill}\\

We begin by proving that if $L$ is an invertible linear preserver of semipositive 
matrices that also maps every rank one semipositive matrix to a rank one (semipositive) matrix, 
then $L$ is in the standard form. It is obvious that 
if $L(A) = XAY$ for some invertible row positive matrix $X$ and inverse 
nonnegative matrix $Y$, the map $L$ has the property mentioned above.

\medskip
\noindent
\subsubsection{\textbf{A sufficient condition}} \hspace{\fill}\\

For $0 \neq x \in \reals^m$ and consider the set $U_x := \{xu^t: u \in \reals^n\}$, 
an $n$-dimensional subspace of $M_{m,n}$ consisting of matrices rank at most one.

\begin{theorem}\label{theorem-suff-cond}
Let $L: M_{m,n} \rightarrow M_{m,n}$  be an invertible linear map. Assume 
that $L$ satisfies the following conditions:
\begin{enumerate}
\item $L$ is an into preserver of $S(\reals^n_+,\reals^m_+)$.
\item $rank(L(A)) = 1$, whenever $rank(A) =1$ and $A \in S(\reals^n_+,\reals^m_+)$.
\end{enumerate}
Then $L(A) = XAY$ for any $A \in M_{m,n}$, where $X \in M_m$ is invertible 
and row positive and $Y \in M_n $ is inverse nonnegative.
\end{theorem}

\begin{proof}
The proof involves four steps.
	
\medskip
\noindent
\underline{Claim 1:} For any $x \in (\reals^m_+)^\circ \cup (-(\reals^m_+)^\circ)$, any 
$0 \neq A \in L(U_x)$ has rank one.
	
\medskip
\noindent
\underline{Proof:} Let $x \in (\reals^m_+)^\circ$. We discuss two cases here. \\
\noindent
\underline{Case 1:} Suppose $u \in \reals^n \setminus (-\reals^n_+)$. Then 
$xu^t \in S(\reals^n_+,\reals^m_+)$ and so $rank(L(xu^t)) =1$. \\
\noindent
\underline{Case 2:} $0 \neq u \in -\reals^n_+$. In this case, 
$x(- u)^t \in S (\reals^n_+,\reals^m_+)$, and once again $rank(L(x(-u)^t)) =1$. 
A similar conclusion holds for $x \in (-\reals^m_+)^\circ$. 
	
\vspace{0.2cm}
\noindent
\underline{Claim 2:} For $x \in (\reals^m_+)^\circ \cup (-(\reals^m_+)^\circ), 
\ L(U_x) = U_y$  for some $y \in \reals^m$.
	
\noindent
\underline{Proof:} Suppose the claim is not true. Without loss of generality, assume 
that there exist linearly independent vectors $u_1, u_2 \in \reals^n$ such that 
$xu_1^t $, $xu_2^t \in S(\reals^n_+,\reals^m_+),  L(xu_1^t) = z_1 q_1^t$ and 
$L(xu_2^t) = z_2 q_2^t$, where $z_1, z_2 \in \reals^n_+$ are linearly independent and 
$q_1, q_2 \in \reals^m$. We then have $L(x(u_1 + u_2)^t) = L(xu_1^t) + L(xu_2^t) = 
z_1 q_1^t + z_2 q_2^t$. Since $rank(x(u_1+ u_2)^t) = 1$, we have $L(x(u_1+u_2)^t) = 
z_3 q_3^t$. Therefore, $q_1$ and $q_2$ must be linearly dependent. 
We thus have a $2$-dimensional subspace of $M_{m,n}$ that is mapped to a 
$1$-dimensional subspace of $M_{m,n}$. This contradiction proves the claim. 
	
\medskip
\noindent
\underline{Claim 3:} For $0 \neq z \in \reals^m, \ L(U_z) = U_p$ for $0 \neq p \in \reals^m$. 
	
\noindent
\underline{Proof:} If $z \in (\reals^m_+)^\circ \cup (-(\reals^m_+)^\circ)$, then we 
have $L(U_z) = U_p$, for some $p \in \reals^m$ (by Claim 2). If $z$ belongs to the complement 
of $(\reals^m_+)^\circ \cup (-(\reals^m_+)^\circ)$, then we can write 
$z = z_1- z_2$, for some linearly independent $z_1, z_2 \in (\reals^m_+)^\circ$. 
Since $z_1$ and $z_2$ are linearly independent, 
there exist linearly independent $p_1$ and $p_2$ such that $L(U_{z_1}) = U_{p_1}$ 
and $L(U_{z_2}) = U_{p_2}$. Suppose there exists $q \in \reals^n$ such that 
$L(z_1q^t)= p_1q_1^t$ and $L(z_2q^t)= p_2q_2^t$, where $q_1$ and $q_2$ are linearly 
independent. We then have $L((z_1 + z_2)q^t) = p_1q_1^t + p_2q_2^t$, a rank two matrix. 
This contradiction (using Claim 1) proves that $L(U_z) = U_p$, for some $0 \neq p \in \reals^m$.
	
\medskip
\noindent
We have thus proved that $L$ preserves the set of rank $1$ matrices in $M_{m,n}$. 
It now follows from Theorem 2 of \cite{lau} that $L(A) = XAY$ or $m = n$ and 
$L(A) = XA^tY$, for all $A \in M_{m,n}$, for invertible matrices $X$ and $Y$. Since 
$A \mapsto A^t$ need not preserve semipositivity, there exist no invertible matrices 
$X$ and $Y$ such that the map $A \mapsto XA^tY$ preserves semipositivity. Therefore, 
the map $A \mapsto XA^tY$ can be ruled out. Finally, Theorem 2.4 of \cite{dgjjt}, 
yields the desired conclusion on $X$ and $Y$. 
\end{proof}

Theorem 2 of \cite{lau} is actually a real version of the Marcus-Moyls result on 
rank one preservers (see Theorem 1 and the Corollary following it in 
\cite{marcus-moyls-1}). The idea behind Theorem \ref{theorem-suff-cond} comes from 
\cite{rodman-semrl}. It gives us a sufficient condition to check for a map to preserve 
semipositivity. Note that Theorem \ref{theorem-suff-cond} also holds for $m < n$. 
We prove that if $L$ is an invertible 
linear map on $M_{m,n}$ that preserves semipositivity, then any rank one 
semipositive matrix of the form $A_i = \textbf{xy}_{i}^t$, where 
$\textbf{x} = (x_1, x_2, \ldots, x_m)^t \in \reals^m$, and 
$\textbf{y}_{i} = (0,\ldots, y_i,  \ldots, 0 )^t \in \reals^n$, gets mapped to a 
rank one matrix. We exploit this to prove the result for $n = 2$.

\medskip
\noindent
\subsubsection{\textbf{The $n = 2$ case}} \hspace{\fill} \\

Let us observe that a $2 \times 2$ matrix $A$ is semipositive if and only if $A$ has 
a positive column or has one of the forms 
\begin{equation*}
A = \begin{bmatrix*}[r]
     a & -b\\
     -c & d
    \end{bmatrix*} \ \mbox{or} \  
A = \begin{bmatrix*}[r]
     -b & a\\
      d & -c
    \end{bmatrix*},
\end{equation*} where $a > 0, d > 0, b \geq 0, c \geq 0$ and $ad - bc > 0$. 

Let us take the usual basis $\{E_{ij}: i = 1,2, j = 1,2\}$ of $M_2$. 
For a linear map $L$ on $M_2$, let us write down the matrix representation of 
$L(E_{ij})$. It is then easy to write the matrix representation of any rank one 
matrix $A = \textbf{xy}^t$. We do this below.

\medskip
\noindent
Let $L$ be a linear map on $M_2$ and let $A = \textbf{xy}^t$ be a rank one matrix, 
where $\textbf{x} = (x_1,x_2)^t$ and $\textbf{y} = (y_1,y_2)^t$. We then have 
\begin{equation*}L(A) =
\begin{bmatrix*}[r]
(\alpha_1 x_1 + \alpha_3 x_2)y_1 + (\alpha_2 x_1 + \alpha_4 x_2)y_2 & 
(\beta_1 x_1 + \beta_3 x_2)y_1 + (\beta_2 x_1 + \beta_4 x_2)y_2 \\
(\gamma_1 x_1 + \gamma_3 x_2)y_1 + (\gamma_2 x_1 + \gamma_4 x_2)y_2 & 
(\delta_1 x_1 + \delta_3 x_2)y_1 + (\delta_2 x_1 + \delta_4 x_2)y_2
\end{bmatrix*}, 
\end{equation*}
where $\alpha_i, \beta_i, \gamma_i, \delta_i, \ i = 1, \ldots, 4$ are fixed real numbers. 
In other words, we have \\
$L(A) =
\begin{bmatrix*}[r]
(\alpha_1 x_1 + \alpha_3 x_2) & (\beta_1 x_1 + \beta_3 x_2)\\
(\gamma_1 x_1 + \gamma_3 x_2) & (\delta_1 x_1 + \delta_3 x_2)
\end{bmatrix*} y_1 +$  
$\begin{bmatrix*}[r]
(\alpha_2 x_1 + \alpha_4 x_2)  & (\beta_2 x_1 + \beta_4 x_2) \\
(\gamma_2 x_1 + \gamma_4 x_2) & (\delta_2 x_1 + \delta_4 x_2)
\end{bmatrix*} y_2$. 
\medskip

\noindent
A similar form exists for $n \geq 3$ that will be used later.

\begin{theorem}\label{2x2-thm-1}
Let $L$ be an invertible linear map on $M_2$ and $L(S(\reals^2_+)) \subset S(\reals^2_+)$. 
If $A_1 = \textbf{xy}_{1}^t \in S(\reals^2_+)$ and $A_2 = \textbf{xy}_{2}^t \in S(\reals^2_+)$, 
where $\textbf{x} = (x_1, x_2)^t, \textbf{y}_{1} = (y_1, 0)^t$ and $\textbf{y}_{2} = (0, y_2)^t$, 
then \textit{rank} $(L(A_1)) = 1$ and \textit{rank} $(L(A_2)) = 1$. Moreover, 
$L(A_1) = \textbf{uv}^t y_1$ and $L(A_2) = \textbf{pq}^t y_2$, 
where $\textbf{u} =( (\alpha_1 x_1 + \alpha_3 x_2), (\gamma_1 x_1 + \gamma_3 x_2) )^t, \ 
\textbf{v} = (1, -\alpha)^t$, $\textbf{p} = ((\beta_2 x_1 + \beta_4 x_2), (\delta_2 x_1 + \delta_4 x_2))^t, \ 
\textbf{q} = (-\gamma, 1)^t$, 
$\begin{bmatrix*}[r]
\alpha_1 & \alpha_3 \\
\gamma_1 & \gamma_3
\end{bmatrix*}  \geq 0 \ 
\begin{bmatrix*}[r]
\beta_2 & \beta_4 \\
\delta_2 & \delta_4
\end{bmatrix*} \geq 0, \ \alpha \geq 0 \ \& \ \gamma \geq 0$.
\end{theorem}

\begin{proof}
The proof involves several steps.
 
\noindent
\underline{Claim 1:} $L(A_1)$ and $L(A_2)$ are not minimally semipositive.

\noindent
Suppose there exists a $A_1 = \textbf{xy}_{1}^t \in S(\reals^2_+)$ such that $L(A_1)$ is 
a minimally semipositive matrix. We know that either $L(A_1) = 
\begin{bmatrix*}[r]
a & -b \\
-c & d
\end{bmatrix*} y_1$ or $L(A_1) = 
\begin{bmatrix*}[r]
-b & a \\
d & -c
\end{bmatrix*} y_1$, where $a > 0$, $d > 0$, $b \geq 0$, $c \geq 0$ and $ad > bc$. 
Consider the matrix $B = \begin{bmatrix*}[r]
						  -x_1 & x_1 y_2  \\
						  -x_2 & x_2 y_2
						  \end{bmatrix*}$, where $y_2 > 0$. It is clear that $B$ is 
semipositive. If the inverse of $L\Bigg(\begin{bmatrix*}[r]
						  -x_1 & 0  \\
						  -x_2 & 0
						  \end{bmatrix*}\Bigg)$ is negative, it is possible to choose 
a $y_2 > 0$, sufficiently small, such that the inverse of $L(B)$ is nonpositive. It 
follows that $L(B) \notin S(\reals^2_+)$. If the inverse of 
$L\Bigg(\begin{bmatrix*}[r]
						  -x_1 & 0  \\
						  -x_2 & 0
						  \end{bmatrix*}\Bigg)$ is nonpositive, say 
$L\Bigg(\begin{bmatrix*}[r]
						  -x_1 & 0  \\
						  -x_2 & 0
						  \end{bmatrix*}\Bigg) = 
						 \begin{bmatrix*}[r]
						  -a & b\\
						  0 & -d
						  \end{bmatrix*}$ with $a >0, \ b > 0, \ d > 0$, then 
$L(B) = \begin{bmatrix*}[r]
		 -a & b\\
		  0 & -d
		\end{bmatrix*}	+ \begin{bmatrix*}[r]
						  {\ast} & {\ast}\\
						  f & {\ast}
						  \end{bmatrix*}$. If $f \geq 0$, then choose a sufficiently small $y_2 > 0$ 
such that the inverse of $L(B)$ is nonpositive, thereby making 
$L(B)$ not semipositive. If $f < 0$, then choose $y_2 > 0$ such that the second row of 
$L(B)$ is negative. This once again makes $L(B)$ not semipositive. 			  
This proves the claim. Thus, $L(A_1)$ must be a redundantly semipositive matrix.
 
\medskip
		
\noindent
\underline{Claim 2:} $L(A_1)$ cannot have a positive row. 
	
\noindent
If $L(A_1)$ has positive row, then by choosing $y_2 > 0$ sufficiently small and by taking 
$B = \begin{bmatrix*}[r]
-x_1 & x_1 y_2  \\
-x_2 & x_2 y_2
\end{bmatrix*} \in S(\reals^2_+)$, we can show that $L(B)$ contains nonpositive row and 
consequently will not be semipositive.\\
	
We can thus assume without loss of generality that, 
$L(A_1) = \begin{bmatrix*}[r]
a & -b \\
c & -d
\end{bmatrix*}y_1$, where $a>0, c > 0, b \geq 0, d \geq 0$. Thus, 
$L(A_1) = \begin{bmatrix*}[r]
\alpha_1 x_1 + \alpha_3 x_2 & \beta_1 x_1 + \beta_3 x_2 \\
\gamma_1 x_1 + \gamma_3 x_2 & \delta_1 x_1 + \delta_3 x_2
\end{bmatrix*} y_1$. 
	
\medskip
	
\noindent
\underline{Claim 3:} $(\alpha_1,\alpha_3)^t$ and $(\beta_1,\beta_3)^t$ are linearly dependent. 
	
\noindent
Suppose $(\alpha_1,\alpha_3)^t$ and $(\beta_1,\beta_3)^t$ are linearly independent. 
Consider the invertible matrix $W = \begin{bmatrix*}[r]
\alpha_1 & \alpha_3 \\
\beta_1 & \beta_3
\end{bmatrix*}$ and let $d= (-d_1, -d_2)^t < 0$. Take $B = \begin{bmatrix*}[r]
p_1 & x_1 y_2 \\
p_2 & x_2 y_2
\end{bmatrix*} \in S(\reals^2_+)$, where $p = (p_1, p_2)^t$ is such that $Wp = d$, 
$x_1 > 0$, $x_2 > 0$ and $y_2 > 0$ is sufficiently small enough. It is then possible to 
make the first row of $L(B)$ is negative, thereby making $L(B)$ not semipositive. This 
contradiction proves the claim. 

Since $\alpha_1 x_1 + \alpha_3 x_2 = a$ and $\beta_1 x_1 + \beta_3 x_2 = -b$, there exists 
$\alpha \geq 0$ such that $(\beta_1, \beta_3)^t = -\alpha (\alpha_1, \alpha_3)^t$. 
Similarly, we can show that $(\delta_1, \delta_3)^t = -\beta (\gamma_1, \gamma_3)^t$, 
where $\beta \geq 0$. 
	
\medskip
	
\noindent
\underline{Claim 4:} The matrix $\begin{bmatrix*}[r]
\alpha_1 & \alpha_3 \\
\gamma_1 & \gamma_3
\end{bmatrix*}  \geq 0$. 
	
\noindent
If $(\alpha_1, \alpha_3)^t$ contains both positive and negative entries, then there 
exists $(z_1, z_2)^t > 0$ with $(\alpha_1, \alpha_3)(z_1, z_2)^t = 0$. Then, by taking 
the semipositive matrix $B = \begin{bmatrix*}[r]
z_1 & 0\\
z_2 & 0
\end{bmatrix*}$, we see that $L(B)$ contains a zero row. Thus, 
$L(B) \notin S(\reals^2_+)$, which implies that $(\alpha_1, \alpha_3)^t \geq 0$. 
Similarly,  we can show that $(\gamma_1, \gamma_3)^t \geq 0$. 
	
\medskip
	
\noindent
\underline{Claim 5:} $\alpha = \beta$.
	
\noindent
Suppose $\alpha \neq \beta$. We first consider the $\beta > \alpha$ case; the other 
case is similar.

\noindent
\underline{Case 1:} $\beta - \alpha > 0$. 
	
\noindent
Since $L$ is invertible, the matrix 
$V=\begin{bmatrix*}[r]
\alpha_1 & \alpha_3 \\
\gamma_1 & \gamma_3
\end{bmatrix*}$ is invertible. Let $(q_1, q_2)^t \in \reals^2$ such that 
$V(q_1, q_2)^t = (-1, 1)^t$. Let
$B = \begin{bmatrix*}[r]
q_1 & x_1 y_2 \\
q_2 & x_2 y_2
\end{bmatrix*} \in S(\reals^2_+)$, where $x_1$ and $x_2$ are positive and $y_2 > 0$. 
As in Claim 1, we can choose $y_2 > 0$ that is sufficiently small such that $L(B)$ 
either has a nonpositive inverse or a nonpositive row, thereby making it 
not semipositive. Thus, this case does not arise. 
	
\noindent
\underline{Case 2:} $\beta - \alpha < 0$. This can be dealt with similarly as 
in Case 1. 

\smallskip
\noindent
Hence, we have $L(A_1) = \textbf{uv}^t y_1$. 
	
By the previous argument we can show that either $L(A_2) = \textbf{j} \textbf{k}^t y_2$, 
where $\textbf{j} =((\alpha_2 x_1 + \alpha_4 x_2), (\gamma_2 x_1 + \gamma_4 x_2))^t$ and 
$\textbf{k} = (1, -\gamma)^t$ or $L(A_2) = \textbf{pq}^t y_2$. 
	
\medskip
	
\noindent
\underline{Claim 6:} $L(A_2)$ cannot be in the form $\textbf{j} \textbf{k}^t y_2$. 
	
\noindent
Suppose $L(A_2) =\textbf{jk}^t y_2$. It can be easily seen that $L(A_1 + A_2) = 
\textbf{uv}^t y_1 + \textbf{jk}^t y_2$. As $L$ is invertible, 
$\begin{bmatrix*}[r] 1 & 1 \\ -\alpha & -\gamma \end{bmatrix*}$ is also invertible. 
Let us take $(-d_1 -d_2)^t < 0$ and discuss two cases. 
	
\noindent
\underline{Case 1:} $-\gamma + \alpha > 0$.
	
\noindent
There exists $(y_1, -y_2)^t \in \reals^2$, where $y_1$ and $y_2$ are positive such that 
$(\alpha_1 x_1 + \alpha_3 x_2)(y_1) + (\alpha_2 x_1 + \alpha_4 x_2) (-y_2) = -d_1$ and  
$(\alpha_1 x_1 + \alpha_3 x_2)(y_1)(-\alpha) + (\alpha_2 x_1 + \alpha_4 x_2) (-y_2) (-\gamma) 
= -d_2$. Observe now that the first row of 
$L\Bigg( \begin{bmatrix*}[r]
y_1 x_1 & -y_2 x_1  \\
y_1 x_2 & -y_2 x_2
\end{bmatrix*} \Bigg)$ is negative. 
	
\noindent
\underline{Case 2:} $-\gamma + \alpha < 0$.
	
\noindent
There exists $(-y_1, y_2)^t \in \reals^2$, where $y_1$ and $y_2$ are positive, such that 
$(\alpha_1 x_1 + \alpha_3 x_2)(-y_1) + (\alpha_2 x_1 + \alpha_4 x_2) (y_2) = -d_1$ and  
$(\alpha_1 x_1 + \alpha_3 x_2)(-y_1)(-\alpha) + (\alpha_2 x_1 + 
\alpha_4 x_2) (y_2) (-\gamma) = -d_2$. This makes the first row of 
$L\Bigg( \begin{bmatrix*}[r]
-y_1 x_1 & y_2 x_1  \\
-y_1 x_2 & y_2 x_2
\end{bmatrix*} \Bigg)$ is negative. Thus, $L(A_2)$ must be in the form 
$\textbf{p} \textbf{q}^t y_2$. 
	
\medskip
Combining the above claims, the theorem follows.
\end{proof}

We now prove our main theorem.

\begin{theorem}\label{2x2-main-thm}
Let $L$ be an invertible linear map on $M_2$. If $L$ is an into preserver of 
$S(\reals^2_+)$, then $L(A) = XAY$, for some invertible row positive $X \in M_2$ 
and an inverse nonnegative $Y \in M_2$. 
\end{theorem}

\begin{proof}
By Theorem \ref{2x2-thm-1}, we have $L(A_1) = \begin{bmatrix*}[r]
(\alpha_1 x_1 + \alpha_3 x_2) & -\alpha (\alpha_1 x_1 + \alpha_3 x_2) \\
(\gamma_1 x_1 + \gamma_3 x_2) & -\alpha (\gamma_1 x_1 + \gamma_3 x_2)
\end{bmatrix*} y_1$ and $L(A_2) = \\ 
\begin{bmatrix*}[r]
-\gamma(\beta_2 x_1 + \beta_4 x_2) &  (\beta_2 x_1 + \beta_4 x_2) \\ 
-\gamma(\delta_2 x_1 + \delta_4 x_2) & (\delta_2 x_1 + \delta_4 x_2)
\end{bmatrix*} y_2$, where $\begin{bmatrix*}[r]
\alpha_1 & \alpha_3 \\
\gamma_1 & \gamma_3
\end{bmatrix*}$ and $\begin{bmatrix*}[r]
\beta_2 & \beta_4 \\
\delta_2 & \delta_4
\end{bmatrix*}$ are nonnegative matrices, $\alpha \geq 0$ and $\gamma \geq 0$.
	
\medskip
\noindent
\underline{Claim 1:} $(\alpha_1, \alpha_3)^t$ and $(\beta_2, \beta_4)^t$ are linearly dependent. 
	
\noindent
Suppose not. We consider two cases. Let $B = \begin{bmatrix*}[r]
						-\alpha_3 & \beta_4\\
						\alpha_1 & -\beta_2
					           \end{bmatrix*}$. If $\text{det}(B) < 0$, then 
$B \in S(\reals^2_+)$ and the first row of $L(B)$ is zero. If instead 
$\text{det}(B) > 0$, then $- B \in S(\reals^2_+)$ and the 
first row of $L(- B)$ is zero. Thus,  $(\alpha_1, \alpha_3)^t$ and 
$(\beta_2, \beta_4)^t$ are linearly dependent. 
This proves the claim.

\medskip
\noindent	
Similarly, we can show that $(\gamma_1, \gamma_3)^t$ and $(\delta_2, \delta_4)^t$
are linearly dependent. Notice that 
$(\beta_2, \beta_4)^t = \theta (\alpha_1, \alpha_3)^t$ 
and $(\delta_2, \delta_4)^t = \lambda (\gamma_1, \gamma_3)^t$, where $\theta > 0$ 
and $\lambda > 0$. 
	
\medskip
	
\noindent
\underline{Claim 2:} $\theta = \lambda$. 
	
\noindent
Observe that \\
$L \Bigg( \begin{bmatrix*}[r]
a & b \\
c & d
\end{bmatrix*} \Bigg)  = 
\begin{bmatrix*}[r]
\alpha_1 a + \alpha_3 c -\gamma \theta (\alpha_1 b + \alpha_3 d) &
-\alpha(\alpha_1 a + \alpha_3 c) + \theta (\alpha_1 b + \alpha_3 d) \\
\gamma_1 a + \gamma_3 c -\gamma \lambda (\gamma_1 b + \gamma_3 d) &
-\alpha(\gamma_1 a + \gamma_3 c) + \lambda (\gamma_1 b + \gamma_3 d) 
\end{bmatrix*}$. \\

\noindent	
Suppose $\theta \neq \lambda $. We contradict the $\lambda > \theta$ case; the 
other case is similar. Say $\lambda > \theta$, and let $d_1, d_2 > 0$. We can find 
$(u_1, u_2)^t \in \reals^2$ and $(v_1 \ v_2)^t \in \reals^2$ such that 
$\alpha_1 u_1 + \alpha_3 u_2 = -d_1, \ 
\gamma_1 u_1 + \gamma_3 u_2 = d_2, \ \alpha_1 v_1 + \alpha_3 v_2 = \frac{d_1}{\theta}$ 
\ and $\gamma_1 v_1 + \gamma_3 v_2 = - \frac{d_2}{\lambda}$. It can be easily verified 
that $B = \begin{bmatrix*}[r] u_1 & v_1 \\ u_2 & v_2 \end{bmatrix*}$ is a minimally semipositive matrix. 
However, 
$L(B) = \begin{bmatrix*}[r]
(1+ \gamma)(-d_1) & (1+\alpha)d_1 \\
(1+ \gamma)d_2 & (1+\alpha)(-d_2)
\end{bmatrix*} \notin S(\reals^2_+)$. 
	
Finally, we get $L \Bigg( \begin{bmatrix*}[r]
a & b \\
c & d
\end{bmatrix*} \Bigg) = \begin{bmatrix*}[r]
\alpha_1 & \alpha_3 \\
\gamma_1 & \gamma_3
\end{bmatrix*} \begin{bmatrix*}[r]
a & b \\
c & d
\end{bmatrix*} \begin{bmatrix*}[r]
1 & -\alpha \\
-\gamma \theta & \theta 
\end{bmatrix*}$. 
Observe that $X = \begin{bmatrix*}[r]
\alpha_1 & \alpha_3 \\
\gamma_1 & \gamma_3
\end{bmatrix*}$ is an invertible row positive matrix and $Y = \begin{bmatrix*}[r]
1 & -\alpha \\
-\gamma \theta & \theta 
\end{bmatrix*}$ is inverse nonnegative (see Theorem 2.4, \cite{dgjjt}).
\end{proof}

\begin{remark}\label{L(A_i)-general-case}
Before proceeding further, let us write down the matrix representation of $L(A_i)$ 
in the general case, where $A_i = \textbf{xy}_{i}^t$, with 
$\textbf{x} = (x_1,x_2, \ldots, x_m)^t \in \reals^m$ and 
$\textbf{y}_{i} = (0,\ldots, y_i,  \ldots, 0 )^t \in \reals^n$. Let $L$ be a linear 
map on $M_{m,n}$. We then have 
\begin{equation*}
L(A_i) = 
\begin{bmatrix*}[c]
s_1 & s_2 & \cdots & s_n\\
s_{n+1} & s_{n+2}& \cdots & s_{2n}\\
\vdots & \cdots &\cdots & \vdots \\
s_{(m-1)n+1} & s_{(m-1)n+2}& \cdots & s_{mn}
\end{bmatrix*} y_i,
\end{equation*}
where $s_k =l_{k,i} x_1 + l_{k,n+i}x_2+ \cdots +l_{k,(m-1)n +i} x_m$ and $l_{i,j}, 
\ i = 1, \ldots, mn, \ j =1, \ldots, mn$ are fixed real numbers.	
\end{remark}

The above representation is obtained similar to the $n=2$ case by taking 
the usual basis $E_{ij}$ of $M_{m,n}$.

\begin{remark}\label{mx2}
Letting $n = 2$ in REMARK \ref{L(A_i)-general-case}, we observe that the arguments of 
Theorems \ref{2x2-thm-1} and \ref{2x2-main-thm} carry over for invertible maps on 
$M_{m,2}, m \geq 2$ that preserve semipositivity. Recall that when $m \geq 2$, an 
$m \times 2$ matrix is semipositive if and only if every $2 \times 2$ submatrix is semipositive. 
We thus have the following theorem.
\end{remark}

\begin{theorem}\label{mx2-thm}
Let $L$ be an invertible linear map on $M_{m,2}$, where $m \geq 2$. Then, 
\begin{enumerate}
\item \textit{rank} $(L(A_i)) = 1$, where $A_i$ is as above. 
\item $L(A) = XAY$, for some invertible row positive $X \in M_m$ 
and an inverse nonnegative $Y \in M_2$.
\end{enumerate}
\end{theorem}

\medskip
\noindent
\subsubsection{\textbf{The general case: $m \geq n$}} \hspace{\fill}\\

We assume that $m \geq n$. As in the $2 \times 2$ case, the first step is to prove 
that a rank one semipositive matrix of the form $A_i$ (as described in 
REMARK \ref{L(A_i)-general-case}) does not get mapped 
to a minimally semipositive matrix. For simplicity, we present a proof in the 
$3 \times 3$ case, which extends to any $n \times n$ matrix. The reduction to 
$m \times n$ case follows as an $m \times n$ matrix is semipositive if and only 
if every $n \times n$ submatrix is semipositive. Recall that a square matrix $A$ 
is said to be reducible if it is permutation similar to a matrix of the form 
\begin{equation*}
\begin{bmatrix}
A_{11} & A_{12}\\
0 & A_{22}
\end{bmatrix},
\end{equation*} where $A_{11}$ and $A_{22}$ are square and nonzero matrices. 
$A$ is said to be irreducible if it is not reducible. Two matrices 
$A$ and $B$ are said to be permutation equivalent if there exist permutation 
matrices $P$ and $Q$ such that $A = PBQ$. A square matrix $A$ is said to be 
partially decomposable if it is permutation equivalent to a matrix of the form 
given above. Otherwise $A$ is said to be fully indecomposable. One can show that 
$A$ is fully indecomposable if and only if $PA$ is irreducible for every permutation 
matrix $P$. If $A$ is partially decomposable, then $A$ is permutation equivalent to 
\begin{equation*}
\begin{bmatrix}
A_{11} & A_{12} & \cdots & A_{1k}\\
0 & A_{22} & \cdots & A_{2k}\\
\vdots  & \cdots & \ddots & \vdots\\
0 & \cdots & 0 & A_{kk}
\end{bmatrix}, 
\end{equation*}
where the $A_{ii}$ are either $1 \times 1$ zero matrices or are fully indecomposable. 
We shall use the following notions and results from \cite{crj1}.

\noindent
\begin{itemize}

\item (Corollary 1, \cite{crj1}) If an $n \times n$ sign pattern matrix $B$ is fully 
indecomposable, then the following are equivalent.
\begin{enumerate}
\item $B$ is inverse nonnegative.
\item $B$ is inverse positive.
\end{enumerate}
\item (Theorem 2, \cite{crj1}) Suppose $B$ is an $n \times n$ decomposable 
sign pattern matrix in the following block form 
$B = 
\begin{bmatrix*}[c]
B_{11} & B_{12} & \ldots & B_{1k}\\
0 & B_{22} & \ldots & B_{2k}\\
0 & \ldots & & B_{3k}\\
\vdots & \vdots & & \vdots \\
0 & \ldots & & B_{kk}
\end{bmatrix*}$, where each $B_{ii}$ is square and either fully indecomposable or 
a $1 \times 1$ zero matrix. Then $B$ is inverse nonnegative if and only if 
\begin{enumerate}
\item each (fully indecomposable) sign pattern matrix 
$B_{ii}, \ i =1, \ldots, k$ is inverse nonnegative.
\item no submatrix of the form $\begin{bmatrix*}[r]
                                  B_{i,i+1} & \ldots & B_{ij}
                                 \end{bmatrix*}$ or 
$\begin{bmatrix*}[c]
   B_{ij}\\
   \vdots \\
   B_{j-1,j}
\end{bmatrix*}$ is nonnegative and nonzero, $1 \leq i < j \leq k$.
\end{enumerate}
\end{itemize}

\begin{theorem}\label{non-MSP}
Let $L$ be an invertible linear map on $M_3$ that preserves semipositive matrices. 
Consider the rank one matrix $A_1$. If $A_1$ is semipositive, then 
$L(A_1)$ cannot be a minimally semipositive matrix. 
\end{theorem}

\begin{proof}
Suppose $A_1 = \begin{bmatrix*}[r]
              x_1\\x_2\\x_3
             \end{bmatrix*} \begin{bmatrix*}[r]
                            y_1 & 0 & 0
                           \end{bmatrix*}$ is a rank one semipositive 
matrix, so that $x_i > 0$ for $i = 1,2,3$ and $y_1 > 0$. If $L(A_1)$ is minimally semipositive, 
then it is inverse nonnegative. We then have the following cases.

\medskip
\noindent
\underline{Case 1:} Suppose $L(A_1)$ is fully indecomposable.

\noindent
In this case, it follows from Corollary 1, \cite{crj1} that $L(A_1)$ will have a 
positive inverse. Consider the semipositive matrix 
$B = \begin{bmatrix*}[r]
 x_1\\x_2\\x_3                                                                
\end{bmatrix*} \begin{bmatrix*}[r]
               -1 & y_2 & 0
              \end{bmatrix*}$, where $y_2 > 0$. 
Then $L(B) = L \Bigg(\begin{bmatrix*}[r]
                      x_1\\x_2\\x_3 
                     \end{bmatrix*} \begin{bmatrix*}[r]
                                        -1 & 0 & 0
                                       \end{bmatrix*} \Bigg) + 
L \Bigg(\begin{bmatrix*}[r]
   x_1\\x_2\\x_3 
  \end{bmatrix*} \begin{bmatrix*}[r]
                 0 & y_2 & 0
                \end{bmatrix*} \Bigg)$. The inverse of the first term is negative, 
whereas the second term is semipositive. It is now possible to choose a $y_2$, 
sufficiently small, so that the inverse of $ L(B)$ is nonpositive. This forces 
$L(B) \notin S(\reals^3_+)$.

\medskip
\noindent
\underline{Case 2:} Suppose $L(A_1)$ is partly decomposable and has the form \\ 
$L(A_1) = \begin{bmatrix*}[r]
         a & -b & -f_1\\
        -c & d & -f_2\\
         0 & 0 & e
        \end{bmatrix*}$, where $a, b, c, \ \mbox{and} \ d$ are positive, $f_1 \ 
\mbox{and} \ f_2$ are nonnegative and $ad-bc > 0$. 

\noindent
Consider the matrix $B$ as in Case 1. Choose $y_2 > 0$ and sufficiently small so that 
$L(B) = \begin{bmatrix*}[r]
         -a_1 & b_1 & f_1^{\ast}\\
         c_1 & -d_1 & f_2^{\ast}\\
         {\ast} & {\ast} & -e_1
        \end{bmatrix*}$, where $a_1, b_1, c_1, d_1 \ \mbox{and} \ e_1$ are 
positive and $f_1^{\ast} \ \mbox{and} \ f_2^{\ast}$ are nonnegative. In this case, it 
is easy to check that $L(B)$ cannot be minimally semipositive. 
$L(B)$ cannot be a redundantly semipositive matrix as well, as no $3 \times 2$ 
submatrix is semipositive.

\medskip
\noindent
\underline{Case 3:} Suppose $L(A)$ is partly decomposable and has the form 
$L(A) = \begin{bmatrix*}[r]
         a & -f_1 & -f_2\\
         0 & d & -f_3\\
         0 & 0 & e
        \end{bmatrix*}$, where $a, d, \ \mbox{and} \ e$ are positive and $f_1, 
f_2 \ \mbox{and} \ f_3$ are nonnegative. This case can be dealt with similar to 
Case 2.

\medskip
\noindent
Combining everything, we see that $L(A_1)$ cannot be a minimally semipositive matrix.
\end{proof}

\begin{remark}
It follows from the above proof that $L(A_2)$ and $L(A_3)$ cannot be mapped to minimally 
semipositive matrices as well. Moreover, the above proof works for any $n \geq 4$ with 
appropriate modifications as well as for the rectangular case.
\end{remark}

We have thus proved that no rank one semipositive matrix of the form $A_i$ can be mapped 
to a minimally semipositive matrix. We are now ready to prove our main results for maps 
on $M_{m,n}$, when $m \geq n \geq 3$.

\begin{theorem}\label{mxn-L(A_1)-rank-one-part-1}
Let $A_1$ be the rank one matrix described earlier. If $A_1$ is semipositive and if 
$L$ is an invertible linear map on $M_{m,n}$ that preserves semipositivity, then the 
matrix 
\begin{equation*}
C_1 = \begin{bmatrix*}[c]
       l_{1,1} & l_{1,(n+1)} & & \ldots & l_{1,(m-1)n+1}\\
       l_{2,1} & l_{2,(n+1)} & & \ldots & l_{2,(m-1)n+1}\\
       \vdots & \vdots & & \ldots & \vdots \\
       l_{n,1} & l_{n,(n+1)} & & \ldots & l_{n,(m-1)n+1}     
      \end{bmatrix*}
\end{equation*} has rank one.
\end{theorem}

\begin{proof}
Suppose $C_1$ has rank $n$. For $z = - (z_1, \ldots, z_n)^t < 0$, choose a vector 
$q = (q_1, \ldots, q_m)^t$ such that $C_1 q = z$. Consider the matrix 
$B =  \begin{bmatrix*}[c]
          q_1 & x_1 y_2 & 0 & \ldots & 0\\
          q_2 & x_2 y_2 & 0 & \ldots & 0\\
          \vdots & \vdots & \vdots & \ldots & \vdots \\
          q_m & x_m y_2 & 0 & \ldots & 0
\end{bmatrix*}$, where $y_2 > 0$. Then, $B$ is semipositive as it contains a positive 
column (Recall that $A_1$ is semipositive). We then have 
$L(B) = \begin{bmatrix*}[c]
         -z_1 & -z_2 & \ldots & -z_n\\
         {\ast} & {\ast} & \ldots & {\ast}\\
         \vdots & \vdots & \vdots & \vdots \\
         {\ast} & {\ast} & \ldots & {\ast} 
\end{bmatrix*} \ + \ y_2 L \Bigg (\begin{bmatrix*}[c]
                                  0 & x_1 & 0 & \ldots & 0\\
                                  0 & x_2 & 0 & \ldots & 0\\ 
                                  \vdots & \vdots & \vdots & \ldots & \vdots \\
                                  0 & x_m & 0 & \ldots & 0
                                 \end{bmatrix*} \Bigg)$. 
The second term in the above expression is semipositive as $L$ preserves 
semipositivity. Choosing $y_2$ sufficiently small, it is possible to make $L(B)$ not 
semipositive. Thus, $C_1$ cannot have rank $n$.

\medskip
\noindent
Consider the matrix $C_k = \begin{bmatrix*}[c]
                           l_{1,k} & l_{1,(n+k)} & & \ldots & l_{1,(m-1)n+k}\\
                           l_{2,k} & l_{2,(n+k)} & & \ldots & l_{2,(m-1)n+k}\\
                           \vdots & \vdots & & \ldots & \vdots \\
                           l_{n,k} & l_{n,(n+k)} & & \ldots & l_{n,(m-1)n+k}     
                          \end{bmatrix*}$. Let us consider the $n \times mn$ matrix 
$C = [C_1 | C_2 | \ldots | C_n]$. Since $L$ is an invertible map, $C$ has rank $n$. 
Suppose $C_1$ has rank $n-1$. Assume without loss of generality that $[C_1|C_2]$ has rank $n$. For $z = - (z_1, \ldots, z_n)^t < 0$, choose a $p \in \reals^{2m}$ such that 
$[C_1|C_2] p = z$. 
Consider the semipositive matrix 
$B = \begin{bmatrix*}[c]
     p_1 & p_{m+1} & x_1 y_3 & 0 & \ldots & 0\\
     p_2 & p_{m+2} & x_2 y_3 & 0 & \ldots & 0\\
    \vdots & \vdots & \vdots & \vdots & \ldots & \vdots\\
    p_m & p_{2m} & x_m y_3 & 0 & \ldots & 0
\end{bmatrix*}$, where $y_3 > 0$. By choosing $y_3$ sufficiently small, it 
is possible to make $L(B)$ not semipositive (the argument is similar to the one 
used in the previous step). Thus, $C_1$ cannot have rank $n-1$. Proceeding analogously, 
we see that $C_1$ cannot have rank $n-2$ and so on, thereby proving that the rank of 
$C_1$ is $1$.
\end{proof} 

We now prove that if $L$ is an invertible map on $M_{m,n}$ that preserves 
semipositivity, then $L(A_i)$ has rank one for each $i =1, \ldots , n$.

\begin{theorem}\label{mxn-L(A_1)-rank-one-part-2}
Let $L$ be an invertible map on $M_{m,n}$ that preserves semipositivity. Suppose for 
each $i = 1, \ldots, n, \ A_i$ is semipositive. Then, $L(A_i)$ has rank one for each 
$i =1, \ldots, n$.	
\end{theorem}

\begin{proof}
We will prove the result for $A_1$. A similar argument works for $i = 2, \ldots, n$. 
We indicate the steps below, where each successive step assumes the 
previous one. We have already proved that the theorem 
when $n =2$. Recall that $L(A_1)$ is a redundantly semipositive matrix.  
The argument presented below has been verified for $n = 3$; since the calculations are 
very lengthy, we are not including them here and an explanation is included in the 
appendix when $n = 3$. Before proceeding further, let us denote by $P_1, \ldots, P_m$ 
the following numbers:

\noindent
\begin{itemize}
\item $P_1:= l_{1,1} x_1 + l_{1,n+1} x_2 + \ldots + l_{1,(m-1)n+1} x_m$ 
\item $P_2:= l_{n+1,1} x_1 + l_{n+1,n+1} x_2 + \ldots + l_{n+1,(m-1)n+1} x_m$
\item \ldots 
\item \ldots
\item $P_m:= l_{(m-1)n+1,1} x_1 + l_{(m-1)n+1,n+1} x_2 + \ldots + l_{(m-1)n+1,(m-1)n+1} x_m$.
\end{itemize}

\begin{enumerate}

\item Step 1: The first step is Theorem \ref{mxn-L(A_1)-rank-one-part-1}.

\medskip

\item Step 2: Deduce that $L(A_1)$ contains a positive column. This involves several 
steps and the calculations are involved and lengthy even in the $n = 3$ case. See the Appendix for explanation.

\medskip

\item Step 3: Using the expressions $P_1, \ldots, P_m$, Steps 1 and 2 and assuming 
that the first column of $L(A_1)$ is positive, we write $L(A_1)$ as \\
$L(A_1) = \begin{bmatrix*}[c]
           P_1 & \alpha_{1,1} P_1 & \ldots & \alpha_{1,n-2} P_1 & \alpha_{1,n-1} P_1\\
           P_2 & \alpha_{2,1} P_2 & \ldots & \alpha_{2,n-2} P_2 & \alpha_{2,n-1} P_2\\
           \vdots & \vdots & \ldots & \vdots & \vdots \\
           P_m & \alpha_{m,1} P_m & \ldots & \alpha_{m,n-2} P_m & \alpha_{m,n-1} P_m
         \end{bmatrix*}$, where $\alpha_{i,j} \in \reals$.

\medskip

\item Step 4: Consider the matrix $E:= \begin{bmatrix*}[c]
\alpha_{1,1} & \ldots & \alpha_{1,n-2} & \alpha_{1,n-1}\\
\alpha_{2,1} & \ldots & \alpha_{2,n-2} & \alpha_{2,n-1}\\
\vdots & \vdots & \ldots & \vdots \\
\alpha_{m,1} & \ldots & \alpha_{m,n-2} & \alpha_{m,n-1}
\end{bmatrix*}$. If the matrix $E$ is semipositive, then $L(A_1)$ has rank one. Therefore 
$L(A_1) = \begin{bmatrix*}[c]
           P_1\\ P_2\\ \vdots\\ P_m
          \end{bmatrix*} \begin{bmatrix*}[r]
                         1 & \alpha_{1,1} & \ldots & \alpha_{1,n-1}
                        \end{bmatrix*}$.

\medskip

\item Step 5: If the matrix $- E$ is semipositive, whereas $E$ is not, then again 
$L(A_1)$ has rank one and 
$L(A_1) = \begin{bmatrix*}[c] 
P_1\\ P_2\\ \vdots\\ P_m
\end{bmatrix*} \begin{bmatrix*}[r]
1 & \alpha_{1,1} & \ldots & \alpha_{1,n-1}
\end{bmatrix*}$.

\medskip

\item Step 6: Verify that $L(A_2)$ has a positive column. If both $E$ and $- E$ 
are not semipositive, then $L(A_1)$ has rank one. The proof of this goes as follows. 
Assuming that the 
second column of $L(A_2)$ is positive, we have  
$L(A_2) = \begin{bmatrix*}[c]
\beta_{1,1} Q_1 &  Q_1 & \ldots & \beta_{1,n-1} Q_1 \\
\beta_{2,1} Q_2 & Q_2 & \ldots & \beta_{2,n-1} Q_2 \\
\vdots & \vdots & \ldots & \vdots \\
\beta_{m,1} Q_m & Q_m &\ldots & \beta_{m,n-1} Q_m  
\end{bmatrix*}$. Choose a vector $p \in \reals^m$ such that 
$L \Bigg (\begin{bmatrix*}[r]
p_1 & 0 & \ldots & 0\\
p_2 & 0 & \ldots & 0\\
\vdots & \vdots & \ldots & \vdots\\
p_m & 0 & \ldots & 0
\end{bmatrix*} \Bigg) = \begin{bmatrix*}[c]
1 & \alpha_{1,1} & \alpha_{1,2} & \ldots & \alpha_{1,n-1}\\
1 & \alpha_{2,1} & \alpha_{2,2} & \ldots & \alpha_{2,n-1}\\
\vdots & \vdots & \vdots & \ldots & \vdots\\
1 & \alpha_{m,1} & \alpha_{m,2} & \ldots & \alpha_{m,n-1}
\end{bmatrix*}$. If $A_2$ is semipositive, then the 
matrix 
$\begin{bmatrix*}[c]
-p_1 & 0 & \ldots & 0\\
-p_2 & 0 & \ldots & 0\\
\vdots & \vdots & \ldots & \vdots\\
-p_m & 0 & \ldots & 0
\end{bmatrix*} + A_2$ is semipositive and so is its image under $L$. Note that 
it is is possible to choose a $y_2 > 0$ su that the first column of the image of the 
above matrix under $L$ is negative. 
Let $q \in \reals^m$ be such that 
$L \Bigg(\begin{bmatrix*}[r]
p_1 & 0 & \ldots & 0\\
p_2 & 0 & \ldots & 0\\
\vdots & \vdots & \ldots & \vdots\\
p_m & 0 & \ldots & 0
\end{bmatrix*} - y_2 \begin{bmatrix*}[c]
0 & q_1 & 0 & \ldots & 0\\
0 & q_2 & 0 & \ldots & 0\\
\vdots & \vdots & \vdots & \ldots & 0\\
0 & q_m & \vdots & \ldots & 0
\end{bmatrix*}\Bigg)$ equals  
$\begin{bmatrix*}[c]
1 - \beta_{1,1} y_2 & \alpha_{1,1} - y_2  & \ldots & \alpha_{1,n-1} - \beta_{1,n-1} y_2 \\
1 - \beta_{2,1} y_2 & \alpha_{2,1} - y_2  & \ldots & \alpha_{2,n-1} - \beta_{2,n-1} y_2 \\
\vdots & \vdots & \ldots & \vdots \\
1 - \beta_{m,1} y_2 & \alpha_{m,1} - y_2  & \ldots & \alpha_{m,n-1}- \beta_{m,n-1} y_2 
\end{bmatrix*}$. It is now possible to choose a sufficiently small $y_2 > 0$ so that the first 
column of the above matrix is positive and the matrix \\
$- \begin{bmatrix*}[c]
\alpha_{1,1}- y_2  & \ldots & \alpha_{1,n-1}- \beta_{1,n-1} y_2 \\
\alpha_{2,1} -  y_2  & \ldots & \alpha_{2,n-1}-\beta_{2,n-1} y_2 \\
\vdots  & \ldots & \vdots \\
\alpha_{m,1}-  y_2 & \ldots &  \alpha_{m,n-1}- \beta_{m,n-1} y_2
\end{bmatrix*}$ is semipositive. From the previous case, it follows that 
$\alpha_{1,1} - y_2 = \alpha_{2,1} - y_2 = \ldots = \alpha_{m,1} - y_2$, which 
in turn yields $\alpha_{1,1} = \alpha_{2,1} = \ldots = \alpha_{m,1} = 0$.
Similarly, we can show that $\alpha_{1,i} = \alpha_{2,i} = \ldots = \alpha_{m,i} = 0$ 
for $i= 2 , \cdots, n-1$. We 
finally have $L(A_1) = \begin{bmatrix*}[c]
P_1\\ P_2 \\ \vdots\\ P_m
\end{bmatrix*} \begin{bmatrix*}[r]
1 & 0 & \ldots & 0
\end{bmatrix*}$.
\end{enumerate}

This proves the theorem.
\end{proof}

\medskip

\noindent
We now prove our main result concerning the structure of an into preserver of semipositivity.

\begin{theorem}\label{mxn-main-thm}
For $m \geq n$, let $L$ be an invertible linear map on $M_{m,n}$ such that 
$L(S(\reals^n_+, \reals^m_+)) \subset S(\reals^n_+, \reals^m_+)$. Then, $L(A) = XAY$ 
for all $A \in M_{m,n}$, where $X \in M_m$ is an invertible row positive matrix and 
$Y \in M_n$ is an inverse nonnegative matrix.	
\end{theorem}

\vspace{-0.35 cm}
\begin{proof}
We know from Theorem \ref{mxn-L(A_1)-rank-one-part-2}, that $L(A_i)$ has rank one for 
each $i = 1, \ldots, n$. Suppose the first and second columns of $L(A_1)$ and $L(A_2)$, 
respectively, are positive. Let $L(A_1) = \begin{bmatrix*}[c]
                                             P_1\\P_2\\ \vdots \\ P_m
                                            \end{bmatrix*} 
\begin{bmatrix*}[r]
1 & \alpha_{1,1} & \ldots & \alpha_{1,n-1}
\end{bmatrix*}$ and $L(A_2) = \begin{bmatrix*}[c]
Q_1\\Q_2\\ \vdots \\ Q_m
\end{bmatrix*} \begin{bmatrix*}[r]
\alpha_{2,1} & 1 & \ldots & \alpha_{2,n-1}
\end{bmatrix*}$.

\medskip
\noindent
\underline{Claim 1:} The vectors $(l_{1,1}, l_{1,n+1}, \ldots, l_{1,(m-1)n+1})^t$ and 
$(l_{2,2}, l_{2,n+2}, \ldots, l_{2,(m-1)n+2})^t$ are linearly dependent. If not, then 
there will exist two vectors $u = (u_1, \ldots, u_m)^t$ an $v = (v_1, \ldots, v_m)^t$ 
in $\reals^m$ such that 
$\begin{bmatrix*}[r]
l_{2,2} & l_{2,n+2} & \ldots & l_{2,(m-1)n+2}\\
l_{1,1} & l_{1,n+1} & \ldots & l_{1,(m-1)n+1}
\end{bmatrix*} \begin{bmatrix*}[c]
               u_1 & v_1\\
               u_2 & v_2\\
               \vdots & \vdots\\
               u_m & v_m
              \end{bmatrix*} = \begin{bmatrix*}[r]
                               1 & 0\\
                               0 & 1
                              \end{bmatrix*}$. 
Since $\begin{bmatrix*}[r]
       l_{2,2} & l_{2,n+2} & \ldots & l_{2,(m-1)n+2}\\
       l_{1,1} & l_{1,n+1} & \ldots & l_{1,(m-1)n+1}
      \end{bmatrix*}$ is a rank two nonnegative matrix, we see that the matrix 
$B = \begin{bmatrix*}[c]
      u_1 & v_1 & 0 & \ldots & 0\\
      u_2 & v_2 & 0 & \ldots & 0\\
      \vdots & \vdots & \vdots & \ldots & \vdots\\
      u_m & v_m & 0 & \ldots & 0
    \end{bmatrix*}$ is an $m \times n$ semipositive matrix. However, 
$L(B) = \begin{bmatrix*}[c]
         0 & 0 & \ldots & 0\\
         {\ast} & {\ast} & \ldots & {\ast}\\
         \vdots & \vdots & \ldots & {\ast}\\
         {\ast} & {\ast} & \ldots & {\ast}
        \end{bmatrix*}$ is not semipositive. Therefore, there exists a positive real number 
$\lambda_{1,1}$ such that $(l_{2,2}, l_{2,n+2}, \ldots, l_{2,(m-1)n+2})^t =  
\lambda_{1,1} (l_{1,1}, l_{1,n+1}, \ldots, l_{1,(m-1)n+1})^t$. Proceeding in a similar way, 
it can be shown that 
$L(A_2) = \begin{bmatrix*}[c]
           \lambda_{1,1} P_1\\ \lambda_{2,1} P_2 \\ \vdots \\ \lambda_{m,1} P_m
\end{bmatrix*} \begin{bmatrix*}[r]
               \alpha_{2,1} & 1 & \ldots & \alpha_{2,n-1}
              \end{bmatrix*}$, where $\lambda_{j,1} > 0$ for $j = 1, \ldots, m$. More 
generally, it can be shown that 
$L(A_i) = \begin{bmatrix*}[c]
\lambda_{1,i-1} P_1\\ \lambda_{2,i-1} P_2 \\ \vdots \\ \lambda_{m,i-1} P_m
\end{bmatrix*} \begin{bmatrix*}[r]
\alpha_{i,1} & \alpha_{i,2} & \ldots & \alpha_{i,n}
\end{bmatrix*}$, where some $\alpha_{i,k} > 0, \ i = 3, 4, \ldots, n$.

\medskip
\noindent
\underline{Claim 2:} $\lambda_{1,i-1} = \lambda_{2,i-1} = \ldots = \lambda_{m,i-1}, \ 
i = 2, \ldots, n$.\\
\noindent
We prove that $\lambda_{1,1} = \lambda_{2,1}$ and skip the remaining 
arguments, as the idea is the same.  Assume that $\lambda_{1,1} > \lambda_{2,1}$. 
Choose positive numbers $d_1, d_2$ and form the matrix $B$ defined as \\
$\begin{bmatrix*}[c]
      l_{1,1} & l_{1,n+1} & \ldots & l_{1,(m-1)n+1}\\
      l_{n+1,1} & l_{n+1,n+1} & \ldots & l_{n+1,(m-1)n+1}\\
      \vdots & \vdots & \ldots & \vdots \\
      l_{(m-1)n+1,1} & l_{(m-1)n+1,n+1} & \ldots & l_{(m-1)n+1,(m-1)n+1}
\end{bmatrix*}^{-1} \begin{bmatrix*}[c]
                    d_1 & \frac{-d_1}{\lambda_{1,1}} & 0 & 0 & \ldots & 0\\
                    -d_2 & \frac{d_2}{\lambda_{2,1}} & 0 & 0 & \ldots & 0\\
                    0 & 0 & 1 & 0 & \ldots & 0\\
                    0 & 0 & 0 & 1 & \ldots & 0\\
                    \vdots & \vdots & \vdots & \vdots & \ddots\\
                    0 & 0 & 0 & 0 & \ldots & 1
                  \end{bmatrix*}$. It is clear that $B$ is a semipositive matrix. 
But $L(B) = \\
\begin{bmatrix*}[c]
(1-\alpha_{2,1})d_1 & (\alpha_{1,1}-\alpha_{2,1})d_1 & \ldots & (\alpha_{1,n-1}-\alpha_{2,n-1})d_1\\
-(1-\alpha_{2,1})d_2 & -(\alpha_{1,1}-\alpha_{2,1})d_2 & \ldots & -(\alpha_{1,n-1}-\alpha_{2,n-1})d_2\\
{\ast} & {\ast} & \ldots & {\ast}\\
{\ast} & {\ast} & \ldots & {\ast}\\
\vdots & \vdots & \ldots & \vdots\\
{\ast} & {\ast} & \ldots & {\ast}\\ 
\end{bmatrix*}$, which is not semipositive (notice that each $2 \times 2$ submatrix 
of the matrix formed from the first rows of $L(B)$ is not semipositive). Therefore, 
$\lambda_{1,1} \leq \lambda_{2,1}$. Similarly, it can be proved that 
$\lambda_{1,1} \geq \lambda_{2,1}$. Consequently, 
$\lambda_{1,1} = \lambda_{2,1}$. Proceeding this way, it can be seen that 
$\lambda_{1,i-1} = \lambda_{2,i-1} = \ldots = \lambda_{m,i-1}, \ 
i = 2, \ldots, n$. We finally have $L(A) = XAY$ for any $A \in M_{m,n}$, where 
$X$ and $Y$ are the matrices \\ 
$X = \begin{bmatrix*}[c]
      l_{1,1} & l_{1,n+1} & \ldots & l_{1,(m-1)n+1}\\
      l_{n+1,1} & l_{n+1,n+1} & \ldots & l_{n+1,(m-1)n+1}\\
      \vdots & \vdots & \ldots & \vdots \\
      l_{(m-1)n+1,1} & l_{(m-1)n+1,n+1} & \ldots & l_{(m-1)n+1,(m-1)n+1}
     \end{bmatrix*}$, a row positive matrix and \\
$Y = 
\begin{bmatrix*}[c]
1 & \alpha_{1,1} & \ldots & \alpha_{1,n-1}\\
\lambda_{1,1}\alpha_{2,1} & \lambda_{1,1} & \ldots & \lambda_{1,1}\alpha_{2,n-1}\\
\lambda_{1,2}\alpha_{3,1} & \lambda_{1,2}\alpha_{3,2} & \ldots & \lambda_{1,2}\alpha_{3,n}\\
\vdots & \vdots & \ldots & \vdots\\
\lambda_{1,n-1}\alpha_{n,1} & \lambda_{1,n-1}\alpha_{n,2} & \ldots & \lambda_{1,n-1}\alpha_{n,n}
\end{bmatrix*}$. Since $L$ preserves semipositivity, it follows from Theorem 2.4 of \cite{dgjjt} 
that $Y$ is inverse nonnegative.
\end{proof}

Summarising everything, we have proved the following theorem.

\begin{theorem}\label{conclusion}
Let $L$ be an invertible linear map on $M_{m,n}$ such that $L(S(\reals^n_+,\reals^m_+))  
\subset S(\reals^n_+,\reals^m_+)$ ($m \geq n \geq 2$). Then, $L(A) = XAY$ for all 
$A \in M_{m,n}$, where $X \in M_m$ is an invertible row positive matrix and 
$Y \in M_n$ is an inverse nonnegative matrix if and only if each $A_i$ that is 
semipositive gets mapped to a rank one matrix. 
\end{theorem}

\subsection{The proper cones case} \hspace{\fill}\\

We begin this section with useful results needed in subsequent sections. These 
include the preserver properties of $S(K_1,K_2)$ and $MS(K_1,K_2)$ under a specific 
map and the existence of a basis for $M_{m,n}$ from either of the above sets. Recall 
that all the cones are assumed to be proper.

\bigskip
\noindent
\subsubsection{\textbf{Useful results}} \hspace{\fill}\\

We begin with the following result known as a Theorem of the Alternative.

\begin{theorem}(Theorem 2.8, \cite{chs})\label{alternative}
For proper cones $K_1$ and $K_2$ in $\reals^n$ and $\reals^m$, respectively, and 
an $m \times n$ matrix $A$, one and only one of the following
alternatives holds.\\
(a) There exists $x \in K_1$ such that $A x \in K_2^\circ$. \\
(b) There exists $0 \neq y \in K_2^*$ such that $- A^t y \in K_1^*$.
\end{theorem}

It is a fairly well known result that the closure of the interior of a convex subset 
$K$ of $\reals^n$ equals the closure of $K$. We shall use this in the proofs later on. 
For completeness, we present a proof.

\begin{lemma}\label{a density argument}
Let $F$ be a convex set in $\reals^n$ with a nonempty  interior. Then, 
$\overline{F^{\circ}} = \overline{F}$.
\end{lemma}	

\begin{proof}
We only need to prove that $\overline{F} \subset \overline{F^{\circ}}$. If $a \in F, 
b \in \overline{F}$, then the set $\{(1/n)a + (1-1/n)b: n \in \mathbb{N}\}$ is contained 
in the interior of $F$. Moreover, elements of the above set converge to $b$. Thus, 
$b \in \overline{F^{\circ}}$.
\end{proof}

The following is a well known fact concerning nonnegative matrices.

\begin{lemma} \label{(K_1,K_2)-nonnegative}
(Corollary $3.3$, \cite{fhp})
Let $K_1$ and $K_2$ be proper cones in $\mathbb{R}^n$ and $\mathbb{R}^m$, respectively, 
and $S : \mathbb{R}^n \rightarrow \mathbb{R}^m$ be a linear map such that 
$S(K_1) \subseteq K_2$. Then $S^t (K_2^*) \subseteq K_1^*$.
\end{lemma}

We now prove that given any element $v$ of a proper cone $K$, there is a subcone 
$K_1$ of $K$ which is simplicial and containing the point $v$.

\begin{lemma}\label{lemma-6}
Let $K$ be a proper cone in $\reals^n$ and $v \in K$. Then there exists an invertible 
$T \in \pi (\reals^n_+, K)$ such that $Tx = v$ for some $x \in \reals^n_+$.
\end{lemma}

\begin{proof}
If $n=2$, then $K$ is a simplicial cone. Therefore $T(\reals^2_+)= K$ for some invertible 
$T \in M_2(\reals)$. In such a case, the result is obvious.

Let $n \geq 3$ and $v \in K$. Since $K$ is a proper cone, there exits $v_2 \in K$ which 
is linearly independent of $v$. Suppose for every $z \in K$, $z = \alpha v_2 + \beta v$ 
for some $\alpha , \beta \in \reals$. Then $K \subsetneq span \{ v , v_2  \}$, a 
proper subspace of $\reals^n$. Since such a subspace has empty interior and 
$K^{\circ} \neq \emptyset$, we get a contradiction. Thus, there exists $v_3 \in K$ such that 
$\{ v, v_2, v_3 \}$ is linearly independent. Proceeding by induction, we get a basis 
$\{ v, v_2, \ldots, v_n \}$ for $\reals^n$ such that $v, v_2, \ldots , v_n \in K$. 
Take  $T = [ v \ v_2 \cdots v_n]$. Then,  $Tx = v$, where $x = [ 1 \ 0 \cdots 0]^t 
\in \reals^n_+$. Since $K$ is a convex cone, $T \in \pi (\reals^n_+, K)$. This completes 
the proof.    
\end{proof}

Let us recall the following result from \cite{csv-1}.

\begin{theorem}\label{cone-nonnegative-sp} 
(Theorem 2.4, \cite{csv-1}) For proper cones $K_1, K_2$ in $\reals^n$, let 
$S \in \pi(K_1,K_2)$ be an invertible linear map on $\reals^n$. If a matrix $A$ is 
$K_1$-semipositive, then the matrix $B = SAS^{-1}$ is $K_2$-semipositive. Conversely, 
if the cones are self-dual and if $C$ is $K_2$-semipositive, then there exists a 
$K_1$-semipositive matrix $A$ such that $C = (S^t)^{-1} A S^t$. 
\end{theorem}

The following will be used subsequently. We state it without proof.

\begin{lemma}\label{lemma-1}
Let $K_1$ and $K_2$ be proper cones in $\mathbb{R}^n$ and $\mathbb{R}^m$, respectively. 
Then the following hold:
\begin{enumerate}
\item Let $Q_1 \in \pi (K_2, \reals^m_+)$ with $Q_1((K_2)^{\circ}) \subseteq 
(\reals^m_+)^{\circ}$ and an invertible $Q_2 \in \pi (K_1, \reals^n_+)$. If 
$A \in S(K_1, K_2)$, then $Q_1 A Q_2^{-1} \in S(\reals^n_+, \reals^m_+)$.

\item Let $S_1 \in \pi (\reals^m_+,K_2)$ with $S_1((\reals^m_+)^{\circ}) \subseteq 
(K_2)^{\circ}$ and an invertible $S_2 \in \pi (\reals^n_+, K_1)$. If 
$B \in S(\reals^n_+, \reals^m_+)$, then $S_1 B S_2^{-1} \in S (K_1, K_2 )$.
\end{enumerate}
\end{lemma}

The following two results are similar to that of Theorem \ref{cone-nonnegative-sp} 
for $MS(K_1,K_2)$.

\begin{lemma}\label{lemma-2}
For $m > n$, let $K_1$ and $K_2$ be proper cones in $\reals^n$ and $\reals^m$, 
respectively, with $K_2$ simplicial. 

\begin{enumerate}
\item If $S$ an $T$ are invertible maps on $\reals^n$ and $\reals^m$ such that 
$S \in \pi (\reals^n_+, K_1)$ and $T (\reals^m_+) = K_2$, then 
$TAS^{-1} \in MS(K_1, K_2)$ whenever $A \in M_{m,n}$ is minimally semipositive. 
\item If $S$ an $T$ are invertible maps on $\reals^n$ and $\reals^m$ such that 
$S \in \pi (\reals^n_+, K_1^*)$ and $T ( \reals^m_+) = K_2$, then 
$T^{-1} B (S^t)^{-1}$ is minimally semipositive whenever $B \in MS(K_1, K_2)$.
\end{enumerate}
\end{lemma}

\begin{proof} 
We prove only the first statement as the proof of the second statement is similar. 
Observe that $TA  \in S(\reals^n_+, K_2)$. Let $B \in M_{n,m}$ be a nonnegative 
left inverse of $A$. Then $SB T^{-1}\in \pi (K_2, K_1)$ is a left inverse of 
$T A S^{-1}$. Therefore, it is enough to prove that  $TA S^{-1} \in S(K_1, K_2)$. 
Suppose $T A S^{-1} \notin S(K_1, K_2)$. By the Theorem of the Alternative 
(Theorem \ref{alternative}), there exists $0 \neq x \in K_2^*$ such that 
$- (S^t)^{-1} A^t T^tx \in K_1^*$. We then have $-A^t T^t x \in \reals^n_+$ which 
implies that $TA \notin S(\reals^n_+, K_2)$, a contradiction. 
\end{proof}

Observe that we need not assume that the cone $K_2$ is simplicial if $m = n$, where 
we have the following result: For proper cones $K_1$ and $K_2$ in $\reals^n$, l
et $S \in \pi(\reals^n_+,K_1)$ and $T \in \pi(\reals^n_+,K_2^{\ast})$ be invertible matrices. Then, $(T^t)^{-1}AS^{-1} \in MS(K_1,K_2)$ whenever $A$ is minimally 
semipositive. Similarly, if $S \in \pi(\reals^n_+,K_1^{\ast})$ and 
$T \in \pi(\reals^n_+,K_2)$ are invertible matrices, then 
$T^{-1}B(S^t)^{-1}$ is minimally semipositive whenever $B \in MS(K_1,K_2)$. 
We skip the proof as it is similar to the above Lemma.

Before proceeding further, let us mention the following useful results that follow 
from Lemmas \ref{lemma-1} and \ref{lemma-2}.

\begin{observation}
For $A \in M_{m,n}$ and proper cones $K_1$ and $K_2$ in $\reals^n$ and 
$\reals^m$, respectively, the following hold:
\begin{enumerate}
\item There exists $B, C \in S(K_1,K_2)$ such that $A = B + C$.
\item There exists $C_1, C_2 \in MS(K_1,K_2)$ such that $A = C_1 - C_2$, 
if in addition the cone $K_2$ is simplicial when $m > n$.
\end{enumerate}
\end{observation}

We end this subsection by proving that $M_{m,n}$ contains a basis from $S(K_1,K_2)$ 
and $MS(K_1,K_2)$. The following result was proved recently by 
P. N. Choudhury et al \cite{prs}.

\begin{theorem} (Theorem 3.1, \cite{prs}) \label{theorem-1}
There is a basis of minimally semipositive matrices for $M_{m,n}, \ m \geq n$. 
\end{theorem}

Below is the proof that $M_{m,n}$ contains a basis from both $S(K_1,K_2)$ as well as 
$MS(K_1,K_2)$.

\begin{theorem}\label{theorem-1.1}
Given proper cones $K_1$ and $K_2$ in $\mathbb{R}^n$ and $\mathbb{R}^m$, 
respectively, the following hold:
\begin{enumerate}
\item $S(K_1, K_2)$ contains a basis for $M_{m,n}$.
		
\item $MS(K_1, K_2)$ contains a basis for $M_{m,n}$, if in addition the cone 
$K_2$ is simplicial, when $m > n$.
\end{enumerate}
\end{theorem} 

\begin{proof}
$(1)$ From \cite{dgjjt}, we know  that $S(\reals^n_+,\reals^m_+)$ contains a basis 
for $M_{m,n}$. Let $\{ A_{ij} \} \subset S (\reals^n_+,\reals^m_+) $ be a basis 
for $M_{m,n}$. Then by Lemma \ref{lemma-1}(2), $\{ B_{ij} = T A_{ij} S^{-1} \} 
\subset S( K_1, K_2)$ will be a basis for $M_{m,n}$, where 
$S \in \pi (\reals^n_+,K_1)$ and $T \in \pi (\reals^m_+,K_2)$ are invertible matrices 
with $T ((\reals^m_+)^{\circ}) 	\subseteq (K_2)^{\circ}$.

\medskip
\noindent 
$(2)$ By Theorem \ref{theorem-1}, let $\{\widetilde{A_{ij}} \}$ be a collection of 
minimally semipositive matrices that form a basis for $M_{m,n}$. Then by 
Lemma \ref{lemma-2}, $\{\widetilde{B_{ij}} = T \widetilde{A_{ij}} S^{-1} \} 
\subset MS( K_1, K_2)$  is a basis for $M_{m,n}$, where $S \in \pi(\reals^n_+,K_1)$ 
and $T (\reals^m_+) = K_2$ are invertible matrices. 
\end{proof}

We are now in a position to tackle preservers of $S(K_1,K_2)$. Recall the following.

\begin{definition}\label{defn-4}
A linear map $L$ on $ M_{m,n}$ is an onto preserver of $\mathcal{S}$ if 
$L(\mathcal{S})= \mathcal{S}$.
\end{definition}

We shall use the following lemma in our proofs (see \cite{dieu} for details).

\begin{lemma}\label{lemma-5}
If $\mathcal{S}$ contains a basis for $M_{m,n}$, then $L$ is an onto preserver 
of $\mathcal{S}$ if and only if $L$ and $L^{-1}$ are into preservers of $\mathcal{S}$. 
\end{lemma} 

\bigskip
\noindent
\subsubsection{\textbf{Preservers of $S(K_1, K_2)$}} \hspace{\fill}\\

We begin this section with results that will be used in the proof of Theorem \ref{theorem-3}. 

\begin{lemma}\label{lemma-7}
Let $X \in M_n$. If $S X T \in \pi (\reals^n_+)$ for all invertible matrices 
$T \in \pi (\reals^n_+, K)$ and $S \in  \pi (K,\reals^n_+)$, then $X \in \pi (K)$.
\end{lemma}

\begin{proof}
Let $v \in K$. By 
Lemma \ref{lemma-6}, there exists $T \in \pi (\reals^n_+, K)$ such that $Tx = v$ 
for some $ x \in \reals^n_+$. We have $S X Tx = S Xv \in \reals^n_+$ for all 
invertible $S \in \pi (K,\reals^n_+)$. Therefore, 
$\langle  S Xv,  u \rangle = \langle  Xv, S^t u \rangle \geq 0$ for all 
$u \in \reals^n_+$ and all invertible $S \in \pi (K,\reals^n_+)$. Let us take 
$p \in K^{\ast}$. By Lemma \ref{lemma-6}, there exists 
an invertible $T_1 \in \pi (\reals^n_+, K^{\ast})$ such that $T_1 y = p$ for some 
$y \in \reals^n_+$. In particular, 
$\langle  Xv, T_1 y \rangle = \langle  Xv, p \rangle \geq 0$. Therefore, we get 
$Xv \in K$.      
\end{proof}

Recall that a square matrix $A$ is said to be row positive if $A$ is nonnegative with a nonzero entry 
in each row.

\begin{lemma}\label{lemma-8}
Let $X \in M_n$. If $S X T$ is row positive for all $T  \in  \pi (\reals^n_+, K)$ 
and $S \in \pi (K, \reals^n_+)$ with $T((\reals^n_+)^{\circ}) \subseteq K^{\circ}$ and 
$S(K^{\circ}) \subseteq (\reals^n_+)^{\circ}$, then $X(K^\circ) \subseteq K^\circ$.
\end{lemma}

\begin{proof}
Let $v \in K^\circ$. Take $T= [\frac{1}{n}v \cdots \frac{1}{n}v]$. We see that 
$T \in \pi (\reals^n_+, K)$ with $T((\reals^n_+)^{\circ}) \subseteq K^{\circ}$ and 
$Tx = v$, where $x =[1 \cdots 1]^t \in (\reals^n_+)^\circ$. We have 
$S X Tx = S Xv \in (\reals^n_+)^\circ$ for all $S  \in  \pi (K,\reals^n_+)$ with 
$S(K^{\circ}) \subseteq (\reals^n_+)^{\circ}$. Then, $\langle  S Xv, u \rangle = 
\langle  Xv, S^t u \rangle > 0$ for all $0 \neq u \in \reals^n_+$ and $S  \in  
\pi(K,\reals^n_+)$ with $S(K^{\circ}) \subseteq (\reals^n_+)^{\circ}$. Let us take 
$0 \neq p \in K^{\ast}$ and $T_1 = [p \ q \cdots q]$, where $q \in (K^{\ast})^\circ$. 
It is easy to verify that $T_1 \in \pi (\reals^n_+, K^{\ast})$ with 
$T((\reals^n_+)^{\circ}) \subseteq K^{\circ}$ and $T_1 y = p$, where 
$ y= [1 \ 0 \cdots 0]^t \in \reals^n_+$. In particular, 
$\langle  Xv, T_1 y \rangle = \langle  Xv, p \rangle  > 0$. Therefore, 
we get $X v \in K^\circ$. 
\end{proof}

The main theorems of this section are proved below.

\begin{theorem}\label{theorem-2}
Let $S_2 \in \pi (\reals^n_+,K_1)$ and $Q_2 \in \pi (K_1, \reals^n_+)$ be invertible 
matrices and $S_1 \in \pi (\reals^m_+,K_2)$ and $Q_1 \in \pi (K_2, \reals^m_+)$ with 
$S_1((\reals^m_+)^{\circ}) \subseteq K_2^{\circ}$ and $Q_1(K_2^{\circ}) \subseteq 
(\reals^m_+)^{\circ}$, respectively. Let $T_1(A)= Q_1 A Q_2^{-1}$ and 
$T_2(A)= S_1 A S_2^{-1}$. If $L$: $M_{m,n} \rightarrow M_{m,n}$ is an into preserver 
of $S(K_1, K_2)$, then $L_1 = T_1L T_2$ is an into preserver of 
$S(\reals^n_+,\reals^m_+)$.
\end{theorem}

\begin{proof}
Let $A \in S (\reals^n_+,\reals^m_+)$. By Lemma \ref{lemma-1}, $T_2(A) \in S(K_1, K_2)$. Then $L T_2 (A) \in S(K_1, K_2) $, since $L(S(K_1, K_2)) \subset S(K_1,K_2)$. 
By Lemma \ref{lemma-1}, we finally have $T_1 L T_2 (A) \in S (\reals^n_+,\reals^m_+)$.
\end{proof}

\begin{remark}
Suppose the map $L$ as well as the matrices $S_1$ and $Q_1$ are invertible 
(so that the maps $T_1$ and $T_2$ are invertible), then the map $L_1$ is an 
invertible linear preserver of semipositivity. It then follows from Theorem 
\ref{mxn-main-thm} that $L_1(A) = XAY$ for every $A \in M_{m,n}$ 
for some invertible row positive matrix $X$ and an inverse nonnegative matrix $Y$. 
This also yields that $L(A) = \widetilde{X} A \widetilde{Y}$ for every $A \in M_{m,n}$ 
for some matrices $\widetilde{X}$ and $\widetilde{Y}$ of appropriate sizes. This gives 
us a motivation to study preserver properties of the map $A \mapsto XAY$ for 
appropriate $X$ and $Y$. We however wish to emphasize that no invertibility assumption 
is made in the following result.
\end{remark}

\begin{theorem}\label{theorem-3}
Let $L(A) = XAY$ be a linear map on $M_{m,n}$, where $X \in M_m 
\ and \ Y \in M_n$ are fixed. $L$ is an into preserver of $S( K_1, K_2)$
if and only if either $X(K_2^\circ) \subseteq K_2^\circ$ and $Y$ is $K_1$-inverse
nonnegative or $ - X(K_2^\circ) \subseteq K_2^\circ$ and $ - Y$ is $K_1$-inverse nonnegative.
\end{theorem}

\begin{proof}
Suppose $X(K_2^\circ) \subseteq K_2^\circ$ and $Y$ is $K_1$-inverse nonnegative,
then $XAY \in S( K_1, K_2)$ whenever $A \in S( K_1, K_2)$. If
$ - X(K_2^\circ) \subseteq K_2^\circ$ and $ - Y$ is $K_1$-inverse nonnegative, 
then $L(A)= (-X)A(-Y) = XAY \in S( K_1, K_2)$ whenever $A \in S(K_1, K_2)$.

Conversely, by Theorem \ref{theorem-2}, the map  $L_1(A) = T_1L T_2(A) = 
Q_1 X S_1 A S_2^{-1} Y Q_2^{-1}$ is an into preserver of $S (\reals^n_+,\reals^m_+)$, 
for all $Q_1, Q_2, S_1, S_2$ (all of them satisfying the assumptions of Theorem 
\ref{theorem-2}). By Theorem 2.4 of \cite{dgjjt}, either $Q_1 X S_1 $ is row positive 
and $S_2^{-1} Y Q_2^{-1}$ is inverse nonnegative or $-Q_1 X S_1 $ is row positive and 
$-S_2^{-1} Y Q_2^{-1}$ is inverse nonnegative. By using Lemmas \ref{lemma-7} and 
\ref{lemma-8}, we finally have either $X(K_2^\circ) \subseteq K_2^\circ$ and $Y$ 
is $K_1$-inverse nonnegative, or $ - X(K_2^\circ) \subseteq K_2^\circ$ and $ - Y$ is
$K_1$-inverse nonnegative. 
\end{proof}

The following corollary follows from Theorems \ref{theorem-1.1}, \ref{theorem-3}, and 
Lemma \ref{lemma-5}.
 
\begin{corollary}\label{cor-1}
The linear map $L(A) = XAY$ is an onto preserver of $S(K_1, K_2)$ if and only if
$X(K_2) = K_2$ and $Y (K_1) =K_1$, or $-X(K_2) = K_2$ and $-Y(K_1) = K_1$.
\end{corollary}

\begin{proof}
We know that  there is a basis for $M_{m,n}$ from $S(K_1, K_2)$. Since $L$ is 
an onto preserver of $S(K_1, K_2)$, both $L$ and its inverse are into preservers
of $S(K_1,K_2)$. Therefore, by the previous theorem, $X(K_2^{\circ}) 
\subseteq K_2^{\circ}$ and $Y(K_1^{\circ}) \subseteq K_1^{\circ}$ 
(or $ -X(K_2^{\circ}) \subseteq K_2^{\circ}$ and $ -Y(K_1^{\circ}) 
\subseteq K_1^{\circ}$). Moreover, $X$ and $Y$ 
(or $-X$ and $-Y$) are $K_2$-inverse nonnegative and $K_1$-inverse 
nonnegative, respectively. Since $K^{\circ}$ is dense in $K$ by 
Lemma \ref{a density argument}, this shows one implication.

Conversely, if $X(K_2) = K_2$ and $Y (K_1) =K_1$, then $L(A) = XAY$
is an into preserver of $S(K_1, K_2)$ and so is $L^{-1}(A) = X^{-1}AY^{-1}$. 
Thus, $L$ is an onto preserver of $S(K_1, K_2)$.
\end{proof}

\bigskip
\noindent
\subsubsection{\textbf{Preservers of $MS(K_1, K_2)$}} \hspace{\fill}\\

We now turn our attention to linear maps $L$ that preserve the set $MS(K_1,K_2)$. We 
start with the following result on nonnegativity.

\begin{lemma}\label{lemma-9}
Let $X \in M_n$. If $S^{-1} X (T^t)^{-1} \in \pi (\reals^n_+)$, for all invertible 
$S \in \pi (\reals^n_+, K)$ and $T \in \pi (\reals^n_+, K^{\ast})$, then $X \in \pi (K)$.
\end{lemma}

\begin{proof}
Let $x \in K$. As $K \subset (T^t)^{-1} (\reals^n_+), \ (T^t)^{-1} x = v$ for some 
$x \in \reals^n_+$. We get 
$\langle S^{-1} Xv, u \rangle = \langle Xv, (S^t)^{-1} u \rangle \geq 0$ 
for all $ u \in \reals^n_+$. Since $K^{\ast} \subset (S^t)^{-1} (\reals^n_+)$, 
it follows that $X v \in K$.
\end{proof}

Our main result is the following.

\begin{theorem}\label{theorem-4}
Let $S_1 \in \pi (\reals^n_+,K_1^*), \ S_2 \in \pi (\reals^m_+,K_2), \ 
Q_1 \in \pi(\reals^n_+,K_1)$ and $Q_2 \in \pi (\reals^m_+,K_2)$ be invertible 
matrices. Assume further that $Q_2(R^m _+) = K_2$ and $S_2(R^m _+) = K_2$. 
Let $P_1(A)= S_2^{-1} A (S_1^t)^{-1}$ and $P_2(A)= Q_2 A Q_1^{-1}$. If 
$L$: $M_{m,n} \rightarrow M_{m,n}$ is an into preserver of $MS(K_1, K_2)$, then 
$L_2 = P_1L P_2$ is an into preserver of minimally semipositive matrices.
\end{theorem}

\begin{proof}
Notice that $L$ and consequently $L_2$ are invertible maps. Let $A$ be minimally semipositive. By Lemma \ref{lemma-2}, $P_2(A) \in MS(K_1, K_2)$. Then 
$L P_2 (A) \in MS(K_1, K_2 )$, 	since $L(MS(K_1, K_2)) \subset MS(K_1,K_2)$. Again, 
by Lemma \ref{lemma-2}, we have $P_1 L P_2 (A)$ is minimally semipositive. 
\end{proof}

Similar to the previous section, we now focus our attention to the map $L(A) = XAY$. 
We have a complete answer in this case too. We discuss the cases $n < m$ and $n = m$ separately.

\begin{theorem}\label{theorem-5}
Let $L(A) = X A Y$ be a linear map on $M_{m,n}$ with $n < m$ (where $n \geq 2$) 
for fixed $X \in M_m \ and \ Y \in M_n$. Then $L$ is an into preserver of 
$MS(K_1, K_2)$ if and only if $X(K_2) = K_2$ and $Y$ is $K_1$-inverse nonnegative, 
or $-X(K_2) = K_2$ and $-Y$ is $K_1$-inverse nonnegative.
\end{theorem}

\begin{proof}
If $A \in MS(K_1, K_2)$ then $A$ has a $(K_2, K_1 )$-nonnegative left inverse. Let 
$B$ be a $(K_2, K_1 )$-nonnegative left inverse for $A$. If $X(K_2) = K_2$ and $Y$ 
is $K_1$-inverse nonnegative, $Y^{-1} B X^{-1}$ is $(K_2, K_1 )$-nonnegative and 
a left inverse for $XAY$. Since $XAY \in S(K_1, K_2)$, it follows that 
$XAY \in MS(K_1, K_2)$. 
	
Conversely, suppose that $L(A)$ is an into preserver of $MS(K_1, K_2)$. By 
Theorem \ref{theorem-4}, $L_2(A) = P_1L P_2(A) = 
(S_2)^{-1} X Q_2 A Q_1^{-1} Y (S_1^t)^{-1}$ is minimally semipositive whenever 
$A$ is minimally semipositive, where $S_1, S_2 , Q_1$ and $Q_2$ satisfy the 
assumptions of Theorem \ref{theorem-4}. By Theorem 2.11 of \cite{dgjjt}, either 
$(S_2)^{-1} X Q_2$ is monomial and $Q_1^{-1} Y (S_1^t)^{-1}$ is inverse nonnegative, 
or $-(S_2)^{-1} X Q_2$ is monomial and $-Q_1^{-1} Y (S_1^t)^{-1}$ is inverse 
nonnegative. The result now follows from Lemmas \ref{lemma-7} and \ref{lemma-9}.
\end{proof}
 
As in the case of onto preservers of $MS(K_1,K_2)$, we have the following corollary. 
We assume again that $n < m$ and that the cone $K_2$ is simplicial.

\begin{corollary}\label{cor-2}
The linear map $L(A) = XAY$ is an onto preserver of $MS(K_1, K_2)$ if and only if
$X(K_2) = K_2$ and $Y (K_1) =K_1$, or $-X(K_2) = K_2$ and $-Y (K_1) = K_1$.
\end{corollary}

\begin{proof}
Suppose the map $L$ is an onto preserver of $MS(K_1, K_2)$. Then L must be invertible 
and both $L$ and $L^{-1}$ are into preservers of $MS(K_1, K_2)$. By the previous 
Theorem, one implication follows.

Conversely, if $X(K_2) = K_2$ and $Y(K_1) = K_1$, then obviously 
$L$ is an onto preserver of $MS(K_1, K_2)$, since for every $B \in MS(K_1, K_2)$, 
we can set $A = X^{-1} B Y^{-1} \in MS(K_1, K_2)$, so that $L(A)= B$ (see also Theorem 
2.11 of \cite{dgjjt}).
\end{proof}

The $n = m$ case is presented below, the into and onto separately. The proof is omitted 
as it is similar to that of Theorem \ref{theorem-5} and follows from Theorem 
\ref{theorem-4}, Lemmas \ref{lemma-7} and \ref{lemma-9} and Theorem 2.10 of \cite{prs-1}. 
Note that we need 
not assume simpliciality of the cone $K_2$ in this case.

\begin{theorem}\label{theorem-5.1}
Let $L(A) = X A Y$ be a linear map on $M_n$ for fixed $X, Y \in M_n$. 
Then $L$ is an into preserver of $MS(K_1, K_2)$ if and only if $X$ is $K_2$-inverse nonnegative and 
$Y$ is $K_1$-inverse nonnegative or $-X$ is $K_2$-inverse nonnegative 
and $-Y$ is $K_1$-inverse nonnegative. 
\end{theorem}

\begin{corollary}\label{cor-2.1}
Let $L(A) = X A Y$ be a linear map on $M_n$ for fixed $X, Y \in M_n$. The map 
$L$ is an onto preserver of $MS(K_1, K_2)$ if and only if $X(K_2) = K_2$ and 
$Y (K_1) =K_1$, or $-X(K_2) = K_2$ and $-Y(K_1) = K_1$.
\end{corollary}

\bigskip
\noindent
\subsubsection{\textbf{General onto preservers of $S(K_1, K_2)$}} \hspace{\fill}\\

We now turn our attention to general onto preservers of $S(K_1,K_2)$. Our main result 
is the following.

\begin{theorem}\label{onto-preservers-proper-cones}
Let $L$ be a linear map on $M_{m,n}$. If $L$ is an onto preserver of $S(K_1,K_2)$, 
then $L(A) = \widetilde{X} A \widetilde{Y}$ for all $A \in M_{m,n}$, where 
$\widetilde{X}(K_2) = K_2$ and $\widetilde{Y}(K_1) = K_1$.	
\end{theorem}

\begin{proof}
By Theorem \ref{theorem-2}, we know that for invertible maps $T_1$ and $T_2$, the map 
$L_1 = T_1LT_2$ is an invertible into linear preserver of $S(\reals^n_+,\reals^m_+)$. 
From Theorem \ref{mxn-main-thm}, we infer that $L_1(A) = XAY$ for an invertible 
row positive matrix $X$ and an inverse nonnegative matrix $Y$. It follows that 
$L(A) = \widetilde{X} A \widetilde{Y}$ for some $\widetilde{X} \in M_m$ and 
$\widetilde{Y} \in M_n$. Finally, Corollary \ref{cor-1} yields the desired conclusion. 
\end{proof}

The following result was proved by A. Chandrashekaran et al.

\begin{theorem}
(Theorem 2.3, \cite{csv-2}) Let $A \in M_{m,n}$ and let $K_1, \ K_2$ be proper 
cones in $\reals^n$ and $\reals^m$, respectively. If $A+B \in S(K_1,K_2)$ for every 
$B \in S(K_1,K_2)$ then $A \in \pi(K_1,K_2)$.
\end{theorem}

We present below connections between onto preservers of $S(K_1,K_2)$ and other 
preserver properties of maps related to $L$. We begin with the following result.

\begin{lemma} \label{lemma-12}
Suppose that $L$ is an onto linear preserver of $S(K_1, K_2)$. Then $L$ is an 
automorphism of the cone $\pi(K_1, K_2)$.
\end{lemma}

\begin{proof}
The proof can be found in Theorem 2.6 of \cite{csv-2}, by making suitable 
modifications.
\end{proof}

We now prove that if $L$ is an onto preserver of $S(K_1,K_2)$, then a map that 
is equivalent to $L$ will be a preserver of $\pi(\reals^n_+,\reals^m_+)$.

\begin{theorem}\label{theorem-8}
Let $S_1 \in \pi (\reals^n_+,K_1)$, $S_2 \in \pi (\reals^n_+,K_1^*)$, $T_1 \in 
\pi(\reals^m_+,K_2^*)$ and $T_2 \in \pi (\reals^m_+,K_2)$ be invertible matrices. 
Let $\widetilde{T_1}(A)= T_1^t A S_1$ and $\widetilde{T_2}(A) = T_2 A S_2^t$.
If $L$: $M_{m,n} \rightarrow M_{m,n}$ is an onto preserver of 
$S(K_1, K_2)$, then $\widetilde{L_1} = \widetilde{T_1}L \widetilde{T_2}$ is an into 
preserver of $\pi (\reals^n_+,\reals^m_+)$.
\end{theorem}

\begin{proof}
Let $A \in \pi (\reals^n_+,\reals^m_+)$. Then 
$\widetilde{T_2} (A) \in \pi ( K_1, K_2)$. By Lemma \ref{lemma-12}, we know that 
$L$ is an onto preserver of $\pi( K_1,K_2)$, so that 
$L \widetilde{T_2} (A) \in \pi (K_1,K_2)$. Hence, 
$\widetilde{L_1} = \widetilde{T_1}L \widetilde{T_2}(A) \in \pi (\reals^n_+,\reals^m_+)$. 
\end{proof}

\begin{remark}
It follows from the above result that when $L$ is an onto preserver of $S(K_1,K_2)$, 
the map $\widetilde{L}$ is also of the form $A \mapsto \widetilde{X} A \widetilde{Y}$, 
for some invertible matrices $\widetilde{X}$ and $\widetilde{Y}$. It can be easily 
seen that $\widetilde{X} = T_1^t X T_2$ and $\widetilde{Y} = S_2^t Y S_1$, which are 
nonnegative with respect to $\reals^m_+$ and $\reals^n_+$, respectively.
\end{remark}

\bigskip
\noindent
\subsubsection{\textbf{Preservers of left semipositivity}} \hspace{\fill}\\

We end with the notion of left semipositivity and a preserver result concerning the 
same.

\begin{definition}
Let $A \in M_{m,n}$. We say $A$ is left $(K_1, K_2)$-semipositive if there exists 
$x \in K_2^{\ast}$ such that $A^tx \in (K_1^{\ast})^\circ$.
\end{definition}

The set of all left $(K_1, K_2 )$-semipositive matrices will be denoted by 
$LS(K_1, K_2)$.

\begin{lemma}\label{lemma-13}
Let $K_1$ and $K_2$ be proper cones in $\mathbb{R}^n$ and $\mathbb{R}^m$, respectively. 
Then the following hold:
\begin{enumerate}
\item Let $Q_1 \in \pi (K_2, \reals^m_+)$  and $Q_2 \in \pi (K_1, \reals^n_+)$ be invertible. If $A \in LS(\reals^n_+, \reals^m_+)$, then 
$Q_1^{-1} A Q_2 \in LS(K_1,K_2))$.

\item Let $S_1 \in \pi (\reals^m_+,K_2)$  and $S_2 \in \pi (\reals^n_+, K_1)$ be 
invertible. If $B \in LS(K_1,K_2))$, then 
$S_1^{-1} B S_2 \in LS(\reals^n_+, \reals^m_+)$.
\end{enumerate}
\end{lemma}

The following theorem can be proved and the proof follows similar to 
Lemma 3.2 of \cite{dgjjt}.

\begin{theorem}
If $L$: $M_{m,n} \rightarrow M_{m,n}$ is an onto preserver of $S(K_1, K_2)$, 
then $L$ is also an onto preserver of $LS(K_1, K_2)$.
\end{theorem}

\begin{theorem}
Let $S_2 \in \pi (\reals^n_+,K_1)$, $Q_2 \in \pi (K_1, \reals^n_+)$, 
$S_1 \in \pi (\reals^m_+,K_2)$ and $Q_1 \in \pi (K_2, \reals^m_+)$ be invertible 
matrices. Let $T_1(A)= Q_1 A Q_2^{-1}$ and $T_2(A)= S_1 A S_2^{-1}$. 
If $L$: $M_{m,n} \rightarrow M_{m,n}$ is an onto preserver of $S(K_1, K_2)$, then 
$L_1 = T_1L T_2$ is an into preserver of $S(\reals^n_+,\reals^m_+)$ and $L_1^{-1}$ 
is an into preserver of $LS(\reals^n_+,\reals^m_+)$.
\end{theorem}

\begin{proof}
By Theorem \ref{theorem-2}, $L_1$ is an into preserver of $S(\reals^n_+,\reals^m_+)$. 
It can be easily seen that $L_1^{-1}$ is an into preserver of 
$LS(\reals^n_+,\reals^m_+)$.  
\end{proof}

\vspace{1cm}
\noindent
\textbf{Acknowledgements:} The authors thank Professors C. R. Johnson, K. C. Sivakumar 
and M. J. Tsatsomeros for their suggestions, comments and encouragement that has 
improved the presentation of the manuscript. The second author acknowledges the 
Council of Scientific and Industrial Research (CSIR), India, for support in the 
form of Junior and Senior Research Fellowships (Award No. 09/997(0033)/2015-EMR-I).

\bibliographystyle{amsplain}

\begin{thebibliography}{99}

\bibitem{csv-1}
A. Chandrashekaran, S. Jayaraman and V. N. Mer, 
\emph{Semipositivity of linear maps relative to proper cones in finite dimensional 
real Hilbert spaces}, Electronic Journal of Linear Algebra, 34 (2018), 304-319.

\bibitem{csv-2}
A. Chandrashekaran, S. Jayaraman and V. N. Mer, 
\emph{A characterization of nonnegativity relative to proper cones}, Indian Journal 
of Pure and Applied Mathematics, 51(3) (2020), 935-944, 
{\em arXiv:1801.09849 v7 [math.FA] 21 May 2019}.

\bibitem{barker-cones}
G. P. Barker, \emph{Theory of cones}, Linear Algebra and its Applications, 39 (1981), 63-291.

\bibitem{bp-book}
A. Berman and R. J. Plemmons, \emph{Nonnegative Matrices in the Mathematical Sciences}, 
SIAM Classics in Applied Mathematics, Philadelphia, 1994.

\bibitem{chs}  
B. Cain, D. Hershkowitz and H. Schneider, \emph{Theorems of the alternative for cones
and Lyapunov regularity of matrices}, Czechoslovak Mathematical Journal, 47(122) (1997),
487-499.

\bibitem{prs}
P. N. Choudhury, R. M. Kannan, and K. C. Sivakumar, \emph{New Contributions to 
Semipositive and Minimally Semipositive Matrices}, Electronic Journal of Linear Algebra, 
34 (2018), 35-53.

\bibitem{prs-1}
P. N. Choudhury, R. M. Kannan, and K. C. Sivakumar, \emph{A note on linear preservers 
of semipositive and minimally semipositive matrices}, Electronic Journal of Linear 
Algebra, 34 (2018), 687-694.

\bibitem{dieu}
J. Dieudonn\`{e}, \emph{Sur un\`{e} g\`{e}n\`{e}ralisation du groupe orthogonal \`{a} 
quatre variables [On a generalization that an orthogonal group has four variables]}, 
Archiv der Mathematik, 1(4) (1948/49), 282-287.

\bibitem{dgjjt}
J. Dorsey, T. Gannon, N. Jacobson, C. R. Johnson and M. Turnansky, \emph{Linear 
preservers of semi-positive matrices}, Linear Multilinear Algebra, 64 (2016), no.9, 1853-1862.

\bibitem{fhp}
E. Haynsworth, M. Fiedler and V. Ptak, \emph{Extreme operators on polyhedral cones}, 
Linear Algebra and its Applications, 13 (1976), 163-172.

\bibitem{crj1}
C. R. Johnson, \emph{Sign patterns of inverse non-negative matrices}, Linear Algebra 
and its Applications, 55 (1983), 69-80.

\bibitem{jks}
C. R. Johnson, M. K. Kerr and D. P. Stanford, \emph{Semipositivity of Matrices}, 
Linear and Multilinear Algebra, 37 (1994), 265-271.

\bibitem{lau}
C. Lautemann, \emph{Linear transformations on matrices: rank preservers and determinant 
preservers (Note)}, Linear and Multilinear Algebra, 10 (1981), 343-345.

\bibitem{lp}
C-K. Li and S. Pierce, \emph{Linear preserver problems}, The American Mathematical 
Monthly, (August-September) (2001), 591--605.

\bibitem{marcus-moyls-1}
M. Marcus and B. N. Moyls, \emph{Transformations on tensor product spaces}, Pacific 
Journal of Mathematics, 9(4) (1959), 1215-1221.

\bibitem{marcus-moyls-2}
M. Marcus and B. N. Moyls, \emph{Linear transformations on algebras of matrices}, 
Canadian Journal of Mathematics, 11 (1959), 61-66.

\bibitem{minc}
H. Minc, \emph{Linear transformations on nonnegative matrices}, Linear Algebra and its 
Applications, 9 (1974), 149-153.

\bibitem{rodman-semrl}
L. Rodman and P. Semrl, \emph{A localization technique for linear preserver problems}, Linear Algebra and its Applications, 433 (2010), 2257-2268.

\bibitem{tsat-2}
M. J. Tsatsomeros, \emph{Geometric mapping properties of semipositive matrices}, Linear 
Algebra and its Applications, 498 (2016), 349-359.

\bibitem{zhang-tang-cao}
X. Zhang, Z. Tang and C. Cao, \emph{Preserver problems on spaces of matrices}, Science Press, Beijing, 2006.

\end{thebibliography}

\section{Appendix}

\noindent
\textbf{Proofs of Theorems \ref{mxn-L(A_1)-rank-one-part-2} and \ref{mxn-main-thm} - 
the $3 \times 3$ case}

Let $L$ be an invertible linear map on $M_3$ and let $A = \textbf{xy}^t$ be a rank one 
matrix. Then, the matrix representation of $L(A)$ can be expressed as follows: Write 
$A = \textbf{xy}^t$, where $\textbf{x} = (x_1,x_2, x_3)^t$ and $\textbf{y} = (y_1,y_2, y_3)^t$. 
We then have \\
$L(A) =
\begin{bmatrix}
l_{11} x_1 + l_{14}x_2 +l_{17} x_3 & l_{21} x_1 + l_{24}x_2 +l_{27} x_3 & l_{31} x_1 + l_{34}x_2 +l_{37} x_3\\
l_{41} x_1 + l_{44}x_2 +l_{47} x_3 & l_{51} x_1 + l_{54}x_2 +l_{57} x_3 & l_{61} x_1 + l_{64}x_2 +l_{67} x_3\\
l_{71} x_1 + l_{74}x_2 +l_{77} x_3 & l_{81} x_1 + l_{84}x_2 +l_{87} x_3 & l_{91} x_1 + l_{94}x_2 +l_{97} x_3\\
\end{bmatrix} y_1  
+$ \\
$ \begin{bmatrix}
l_{12} x_1 + l_{15}x_2 +l_{18} x_3 & l_{22} x_1 + l_{25}x_2 +l_{28} x_3 & l_{32} x_1 + l_{35}x_2 +l_{38} x_3\\
l_{42} x_1 + l_{45}x_2 +l_{48} x_3 & l_{52} x_1 + l_{55}x_2 +l_{58} x_3 & l_{62} x_1 + l_{65}x_2 +l_{68} x_3\\
l_{72} x_1 + l_{75}x_2 +l_{78} x_3 & l_{82} x_1 + l_{85}x_2 +l_{88} x_3 & l_{92} x_1 + l_{95}x_2 +l_{98} x_3\\
\end{bmatrix} y_2
+$ \\
$ \begin{bmatrix}
l_{13} x_1 + l_{16}x_2 +l_{19} x_3 & l_{23} x_1 + l_{26}x_2 +l_{29} x_3 & l_{33} x_1 + l_{36}x_2 +l_{39} x_3\\
l_{43} x_1 + l_{46}x_2 +l_{49} x_3 & l_{53} x_1 + l_{56}x_2 +l_{59} x_3 & l_{63} x_1 + l_{66}x_2 +l_{69} x_3\\
l_{73} x_1 + l_{76}x_2 +l_{79} x_3 & l_{83} x_1 + l_{86}x_2 +l_{89} x_3 & l_{93} x_1 + l_{96}x_2 +l_{99} x_3\\
\end{bmatrix} y_3, $ \\ 
where $l_{ij}, \ i = 1, \ldots, 9, \ j =1, \ldots,9$ are fixed real numbers.

\medskip
\noindent
Our aim is to prove that when $A_1$ is semipositive, $L(A_1)$ is mapped to a 
rank one (semipositive) matrix. The proof involves several steps. Here $A_1$ 
represents a rank one matrix of the form $A_1 = \textbf{xy}^t_{1}$, where 
$\textbf{x} = (x_1,x_2, x_3)^t$ and $\textbf{y}_{1} = (y_1, 0, 0)^t$. We only 
indicate the main steps and include the proofs only 
when necessary. We shall have an occasion to use the following Theorem of the 
Alternative.

\begin{theorem}(Theorem 2.8, \cite{chs})\label{alternative-orthant}
For an $m \times n$ matrix $A$, one and only one of the following
alternatives holds.\\
(a) There exists $x \geq 0$ such that $Ax > 0$. \\
(b) There exists $0 \neq y \geq 0$ such that $- A^t y \geq 0$.
\end{theorem} 

The first step is the following. 

\begin{lemma}\label{C_1-rank-one}
Suppose $A_1$ is semipositive. The matrix 
$C_1 = \begin{bmatrix}
l_{11} & l_{14} & l_{17}\\
l_{21} & l_{24} & l_{27}\\
l_{31} & l_{34} & l_{37}
\end{bmatrix}$ has rank one.
\end{lemma}

The proof is very similar to the one given in Theorem \ref{mxn-L(A_1)-rank-one-part-1}. 
One can also prove that the matrices $\begin{bmatrix*}[c]
                                  l_{41} & l_{44} & l_{47}\\
                                  l_{51} & l_{54} & l_{57}\\
                                  l_{61} & l_{64} & l_{67}
                                 \end{bmatrix*} \ \mbox{and} \ 
\begin{bmatrix*}[c]
l_{71} & l_{74} & l_{77}\\
l_{81} & l_{84} & l_{87}\\
l_{91} & l_{94} & l_{97} 
\end{bmatrix*}$ have rank one. 

The second step is in proving that $L(A_1)$ contains a positive column when $A_1$ is semipositive. This is an important step in the proof. Let us denote by $P_1$, $P_2$ 
and $P_3$ the following numbers: 
$P_1:= l_{11}x_1 + l_{14}x_2 +l_{17}x_3, \ P_2:= l_{41}x_1 + l_{44}x_2 +l_{47}x_3 \ \mbox{and} \  
P_3: = l_{71}x_1 + l_{74}x_2 +l_{77}x_3$.

\begin{lemma}\label{positive-column-existence}
If $L$ is an invertible linear map on $M_3$ that preserves $S(\reals^3_+)$, then 
$L(A_1)$ contains a positive column, when $A_1$ is semipositive.
\end{lemma}

\begin{proof}
Since $L(A_1)$ is redundantly semipositive, assume without loss of generality that the 
submatrix formed from the first two columns of $L(A_1)$ forms a semipositive matrix. 
Suppose 
$L(A_1) = \begin{bmatrix*}[r]
P_1 a & -P_1 b & P_1 f_1\\
-P_2 c & P_2 d & P_2 f_2\\
P_3 e & -P_3 f & P_3 f_3 
\end{bmatrix*}$, where $a > 0, d > 0, e > 0, b \geq 0, c \geq 0, f \geq 0$ and 
$f_1, f_2, f_3 \in \reals$. 
	
\medskip
\noindent
\underline{Suppose $f_3 > 0$:} We discuss various possibilities in this case.

\medskip	
\noindent
\underline{Case 1:} If $f_1 > 0, f_2 > 0$, then $L(A_1)$ contains a positive column.

\medskip	
\noindent
\underline{Case 2:} If $f_1 > 0$ and $f_2 \leq 0$, then $L(A_1)$ has the form 
$\begin{bmatrix*}[r]
P_1 a & -P_1 b & P_1 f_1\\
-P_2 c & P_2 d & -P_2 f_2\\
P_3 e & -P_3 f & P_3 f_3 
\end{bmatrix*}$, where $a > 0, d > 0, e > 0, b \geq 0, c \geq 0, f \geq 0$, 
$f_1 > 0, f_3 > 0$ and $f_2 \geq 0$. Choose a vector $q$ such that $Vq = (-1, -1, 1)^t$, 
where $V = \begin{bmatrix*}[c]
		   l_{11} & l_{14} & l_{17}\\
		   l_{41} & l_{44} & l_{47}\\                                                                                   
           l_{71} & l_{74} & l_{77}
          \end{bmatrix*}$. 
Now consider the matrix $B = \begin{bmatrix}
q_1 & 0 & 0\\
q_2 & 0 & 0\\
q_3 & 0 & 0
\end{bmatrix} \ + \ y_2 \begin{bmatrix}
0 & x_1 & 0\\
0 & x_2 & 0\\
0 & x_3 & 0
\end{bmatrix}$, where $y_2 > 0$. Then,  
$L(B) = \begin{bmatrix*}[r]
-a & b & -f_1\\
c & -d & f_2\\
e & -f & f_3
\end{bmatrix*} \ + \ y_2 L \Bigg (\begin{bmatrix}
0 & x_1 & 0\\
0 & x_2 & 0\\
0 & x_3 & 0
\end{bmatrix} \Bigg)$. If $y_2 > 0$ that is sufficiently small exists such that 
$L(B)$ is not semipositive, we get a contradiction to our assumption. Else, choose 
a vector $q$ such that $Vq = (1, -1, -1)^t$ and proceed as above. Note that in at 
least one of these cases, it is possible to choose $y_2 > 0$ sufficiently small so that $L(B)$ is not semipositive.

\medskip	
\noindent
\underline{Case 3:} Suppose $f_1 \leq 0$ and $f_2 > 0$. Write $L(A_1)$ as 
$\begin{bmatrix*}[r]
P_1 a & -P_2 b & -P_1 f_1\\
-P_2 c & P_2 d & P_2 f_2\\
P_3 e & -P_3 f & P_3 f_3
\end{bmatrix*}$, where $a > 0, d > 0, e > 0, b \geq 0, c \geq 0, f \geq 0$, 
$f_1 \geq 0, f_2 > 0, f_3 > 0$. Choose a vector $q$ such that $Vq = (1, -1, -1)^t$, 
form the semipositive matrix $B$ as in the previous case so that for small enough 
$y_2$ the matrix $L(B)$ is not semipositive.

\medskip	
\noindent
\underline{Case 4:} Suppose $f_1 \leq 0$ and $f_2 \leq 0$. We have 
$L(A_1) = \begin{bmatrix*}[r]
P_1 a & -P_1 b & -P_1 f_1\\
-P_2 c & P_2 d & -P_2 f_2\\
P_3 e & -P_3 f & P_3 f_3
\end{bmatrix*}$, where $a > 0, d > 0, e > 0, b \geq 0, c \geq 0, f \geq 0$, 
$f_1 \geq 0, f_2 \geq 0, f_3 > 0$. If either $b \neq 0$ or $f_1 \neq 0$, choose a 
vector $q$ such that $Vq = (1, -1, -1)^t$ and form the semipositive matrix $B$ as 
in the previous case. Then, $L(B) = \begin{bmatrix*}[r]
a & -b & -f_1\\
c & -d & f_2\\
-e & f & -f_3
\end{bmatrix*} \ + \ y_2 L \Bigg (\begin{bmatrix}
0 & x_1 & 0\\
0 & x_2 & 0\\
0 & x_3 & 0
\end{bmatrix} \Bigg)$. 
Choose a $y_2$ small enough so that $B$ is semipositive, whereas $L(B)$ is not 
semipositive. Suppose both $b$ and $f_1$ are zero. If there is no $y_2 > 0$ such that 
$B$ (as in the above case) is semipositive, whereas $L(B)$ is not, then form the matrix 
$B_1 = \begin{bmatrix}
q_1 & 0 & 0\\
q_2 & 0 & 0\\
q_3 & 0 & 0
\end{bmatrix} \ + \ (-y_2) \begin{bmatrix}
0 & x_1 & 0\\
0 & x_2 & 0\\
0 & x_3 & 0 
\end{bmatrix} \ + \ y_3 \begin{bmatrix}
0 & 0 & x_1\\
0 & 0 & x_2\\
0 & 0 & x_3
\end{bmatrix}$, where 
$y_2 > 0$ is small enough and $y_3 > 0$. Notice that $B_1$ is a semipositive matrix. 
Now choose $y_3>0$ sufficiently small to make $L(B_1)$ not semipositive.
	
\medskip
\noindent
\underline{Suppose $f_3 \leq 0$:} We discuss various possibilities in this case.

\medskip
\noindent
\underline{Case 5:} If $f_1 > 0$ and $f_2 > 0$, then write $L(A_1)$ as 
$\begin{bmatrix*}[r]
P_1 a & -P_1 b & P_1 f_1\\
-P_2 c & P_2 d & P_2 f_2\\
P_3 e & -P_3 f & -P_3 f_3
\end{bmatrix*}$, where $f_1 > 0, f_2 > 0$ and $f_3 \leq 0$. Choose a vector $q$ 
such that $Vq = (-1, -1, 1)^t$ and form the semipositive matrix $B$ as done in 
the previous cases. Then, $L(B) = \begin{bmatrix*}[r]
-a & b & -f_1\\
c & -d & -f_2\\
e & -f & -f_3
\end{bmatrix*} \ + \ y_2 L \Bigg (\begin{bmatrix}
0 & x_1 & 0\\
0 & x_2 & 0\\
0 & x_3 & 0
\end{bmatrix} \Bigg)$. 
Now choose a $y_2$ small enough so that $L(B)$ is not semipositive.

\medskip	
\noindent
\underline{Case 6:} If $f_1 > 0$ and $f_2 \leq 0$, then write $L(A_1)$ as 
$\begin{bmatrix*}[r]
P_1 a & -P_1 b & P_1 f_1\\
-P_2 c & P_2 d & -P_2 f_2\\
P_3 e & -P_3 f & -P_3 f_3
\end{bmatrix*}$, where $f_1 > 0, f_2 \geq 0$ and $f_3 \geq 0$. If $f \neq 0$ or 
$f_3 \neq 0$, then choose a vector $q$ such that $Vq = (-1, -1, 1)^t$. If $B$ denotes 
the semipositive matrix as considered in the previous cases, then $L(B) = 
\begin{bmatrix*}[r]
-a & b & -f_1\\
c & -d & f_2\\
e & -f & -f_3
\end{bmatrix*} \ + \ 
y_2 L \Bigg (\begin{bmatrix}
0 & x_1 & 0\\
0 & x_2 & 0\\
0 & x_3 & 0
\end{bmatrix} \Bigg)$. It is now possible to choose a $y_2 > 0$ that is sufficiently 
small so that $L(B)$ is not semipositive. If $f$ and $f_3$ are both zero and 
if there is no $y_2 > 0$ small enough such that $L(B)$ is not semipositive, then 
consider the matrix $B_1$ from Case $4$ above. 
Then $B_1$ is semipositive. Choose a $y_3>0$ such that $L(B_1)$ is not semipositive.

\medskip
\noindent
\underline{Case 7:} If $f_1 \leq 0$ and $f_2 > 0$, then write $L(A_1)$ as 
$\begin{bmatrix*}[r]
P_1 a & -P_1 b & -P_1 f_1\\
-P_2 c & P_2 d & P_2 f_2\\
P_3 e & -P_3 f & -P_3 f_3
\end{bmatrix*}$, where $f_1 \geq 0, f_2 > 0$ and $f_3 \geq 0$. Choose a vector $q$ 
such that $Vq = (-1, -1, 1)^t$ or $(1, -1, -1)^t$ and form the semipositive matrix 
$B$ as before. Observe that it is possible to choose a $y_2 > 0$ small enough so that 
$L(B)$ is not semipositive in at least one of these cases.

\medskip	
\noindent
\underline{Case 8:} Suppose $f_1 \leq 0$ and $f_2 \leq 0$.

\medskip	
\noindent
\underline{Subcase 8(a):} If $f_3 < 0$, the proceed as in Case 7.

\medskip	
\noindent
\underline{Subcase 8(b):} Suppose $f_3 = 0, f_1 < 0$ and $f_2 \leq 0$. In this case, 
we have 
$L(A_1) = \begin{bmatrix*}[r]
P_1 a & -P_1 b & -P_1 f_1\\
-P_2 c & P_2 d & -P_2 f_2\\
P_3 e & -P_3 f & 0
\end{bmatrix*}$, where $f_1 > 0$ and $f_2 \geq 0$. Choose a vector $q$ 
such that $Vq = (1, -1, -1)^t$, form the semipositive matrix $B$ and choose a 
$y_2 > 0$ sufficiently small such that $L(B)$ is not semipositive.

\medskip	
\noindent
\underline{Subcase 8(c):} Suppose $f_3 = 0, f_1 = 0$ and $f_2 < 0$. In this case, 
we have $L(A_1) = \begin{bmatrix*}[r]
P_1 a & -P_1 b & 0\\
-P_2 c & P_2 d & -P_2 f_2\\
P_3 e & -P_3 f & 0
\end{bmatrix*}$, where $f_2 > 0$. Choose a vector $q$ such that 
$Vq = (-1, -1, -1)^t$ and form the semipositive matrix $B$ as before. If there 
exists a $y_2 > 0$ small enough such that $L(B)$ is not semipositive, we are done. 
Else, choose $y_2 > 0, y_3 > 0$ and form the matrix 
$B_1 = \begin{bmatrix}
q_1 & 0 & 0\\
q_2 & 0 & 0\\
q_3 & 0 & 0
\end{bmatrix} \ + \ (-y_2) \begin{bmatrix}
0 & x_1 & 0\\
0 & x_2 & 0\\
0 & x_3 & 0
\end{bmatrix} \ + \ y_3 \begin{bmatrix}
0 & 0 & x_1\\
0 & 0 & x_2\\
0 & 0 & x_3
\end{bmatrix}$, which is  
semipositive. Choose an appropriate $y_3$ so that $L(B_1)$ is not semipositive.
	
\medskip
\noindent
Combining all these cases, we conclude that $L(A_1)$ contains a positive column. 
\end{proof}

\medskip
\noindent
A similar argument will ensure that $L(A_2)$ and $L(A_3)$ also have positive 
columns when $L$ preserves semipositive matrices and the matrices $A_2$ and $A_3$ are 
semipositive. 

We thus have 
$L(A_1) = \begin{bmatrix*}[c]
		   P_1 & \alpha_1 P_1 & \beta_1 P_1\\
		   P_2 & \alpha_2 P_2 & \beta_2 P_2\\
           P_3 & \alpha_3 P_3 & \beta_3 P_3
         \end{bmatrix*}$ for some $\alpha_i$ and $\beta_i \in \reals$ for $i = 1, 2, 3$.

Before proceeding further, let us denote by $E$ the matrix 
$\begin{bmatrix*}[c]
  \alpha_1  & \beta_1 \\
  \alpha_2  & \beta_2 \\
  \alpha_3  & \beta_3 
 \end{bmatrix*}$. 
 
The next step is the following result. 

\begin{lemma}\label{positive-column-1}
Let $L$ be an invertible linear map on $M_3$ that preserves semipositive matrices. Assume 
that $A_1$ is semipositive, that the first column of $L(A_1)$ is positive and the 
matrix $E$ is semipositive. Then, $L(A_1)$ has rank one.
\end{lemma}

\begin{proof}
The proof involves several steps. 
	
\medskip
\noindent
\underline{Claim 1:} The matrix $E$ cannot contain any $2 \times 2$ submatrix that is minimally semipositive.
	
\medskip
\noindent
\underline{Proof:} Suppose not. Since $E$ is semipositive and there is a $2 \times 2$ submatrix that is minimally semipositive, we may assume without loss of generality 
that it has the form 
$E = \begin{bmatrix}
\alpha_1 & -\beta_1 \\
-\alpha_2 & \beta_2\\
{\ast} & {\ast}
\end{bmatrix}$. Choose a vector $q$ such that $Vq = (-1, -1, -1)^t$ 
and consider the matrix $B = \begin{bmatrix}
q_1 & 0 & 0\\
q_2 & 0 & 0\\
q_3 & 0 & 0
\end{bmatrix} \ + \ y_2 \begin{bmatrix}
0 & x_1 & 0\\
0 & x_2 & 0\\
0 & x_3 & 0
\end{bmatrix}$, where $y_2 > 0$.
Observe that $B$ is semipositive and that 
$L(B) = \begin{bmatrix*}[r]
-1 & -\alpha_1  & \beta_1 \\
-1 & \alpha_2  & -\beta_2 \\
-1 & {\ast} & {\ast}
\end{bmatrix*} + y_2 L \Bigg (\begin{bmatrix}
0 & x_1 & 0 \\
0 & x_2 & 0 \\
0 & x_3 & 0
\end{bmatrix} \Bigg)$. The first term in the above 
expression is not semipositive, 
whereas the second one is. Choosing $y_2$ small enough, it is possible to make 
$L(B)$ not semipositive. This contradiction proves that $E$ cannot contain any 
$2 \times 2$ minimally semipositive submatrix. This also proves that $E$ is a 
redundantly semipositive matrix.
	
\medskip
\noindent
\underline{Claim 2:} Both the columns of $E$ cannot be positive. 
	
\vspace{0.2cm}
\noindent
Suppose both the columns of $E$ are positive. Consider the semipositive matrix 
$B = \begin{bmatrix}
x_1 \\x_2 \\x_3
\end{bmatrix} \begin{bmatrix}
-1 & y_2 & 0
\end{bmatrix}, \ y_2 > 0$. Calculating $L(B)$ as in the previous case 
and by choosing $y_2$ sufficiently small, we conclude that $L(B)$ is not semipositive. 
This contradiction proves the claim.
	
Thus, $L(A_1) = \begin{bmatrix}
P_1 & \alpha_1 P_1 & -\beta_1 P_1\\
P_2 & \alpha_2 P_2 & -\beta_2 P_2\\
P_3 & \alpha_3 P_3 & -\beta_3 P_3
\end{bmatrix}, \ \alpha_i > 0, \beta_i \geq 0 \ \mbox{for} \ i =1, 2, 3$.
	
\noindent
Consider the matrix 
$V = \begin{bmatrix}
l_{11} & l_{14} & l_{17}\\
l_{41} & l_{44} & l_{47}\\
l_{71} & l_{74} & l_{77}
\end{bmatrix}$. It is an invertible matrix, as the map $L$ is invertible. We discuss 
various possibilities in order to prove that $L(A_1)$ is a rank one matrix.
	
\medskip
\noindent
\underline{Suppose $\alpha_1 = \alpha_2$.}
	
\medskip
\noindent
\underline{Case 1:} Assume that $\beta_2 - \beta_1 > 0$. Choose a vector 
$q = (q_1, q_2, q_3)^t \in \reals^3$ such that $Vq = (-1, 1 ,1)^t$. Consider the matrix 
$B = \begin{bmatrix}
q_1 & 0 & 0\\
q_2 & 0 & 0\\
q_3 & 0 & 0
\end{bmatrix} \ + \ y_2 \begin{bmatrix}
0 & x_1 & 0\\
0 & x_2 & 0\\
0 & x_3 & 0
\end{bmatrix}$ with $y_2 > 0$. Then, 
$B$ is semipositive and $L(B) = \begin{bmatrix*}[r]
-1 & -\alpha_1 & \beta_1\\
1 & \alpha_1 & -\beta_2\\
1 & \alpha_3 & -\beta_3
\end{bmatrix*} \ + \ y_2 
L \Bigg (\begin{bmatrix}
0 & x_1 & 0\\
0 & x_2 & 0\\
0 & x_3 & 0
\end{bmatrix} \Bigg)$. By choosing $y_2$ sufficiently small, $L(B)$ can be made 
to be not semipositive. 
	
\medskip
\noindent
\underline{Case 2:} The case $\beta_1 - \beta_2 > 0$ can be handled similarly.
	
\vspace{0.2cm}
\noindent
Thus, $L(A_1) = \begin{bmatrix}
P_1 & \alpha_1 P_1 & -\beta_1 P_1\\
P_2 & \alpha_1 P_2 & -\beta_1 P_2\\
P_3 & \alpha_3 P_3 & -\beta_3 P_3
\end{bmatrix}$.

\medskip
\noindent
\underline{Suppose $\beta_1 = \beta_3 \neq 0$.}
	
\medskip
\noindent
\underline{Case 5:} Suppose $\alpha_1 - \alpha_3 > 0$. Choose $q = (q_1, q_2, q_3)^t$ 
so that 
$Vq = (1, -1, 1)^t$ and consider the matrix $B = \begin{bmatrix}
q_1 & 0 & 0\\
q_2 & 0 & 0\\
q_3 & 0 & 0
\end{bmatrix} \ + \ y_2 \begin{bmatrix}
0 & x_1 & 0\\
0 & x_2 & 0\\
0 & x_3 & 0
\end{bmatrix}$.
If $y_2 > 0$, then $B$ is semipositive. $L(B) = \begin{bmatrix*}[r]
1 & \alpha_1 & -\beta_1\\
-1 & -\alpha_1 & \beta_1\\
1 & \alpha_3 & -\beta_1
\end{bmatrix*} \ + \ 
y_2 L\Bigg (\begin{bmatrix}
0 & x_1 & 0\\
0 & x_2 & 0\\
0 & x_3 & 0
\end{bmatrix} \Bigg)$. The second term is clearly semipositive, whereas the first 
term is not, as $\beta_1 \neq 0$ and $\alpha_1 > \alpha_3$. If $y_2$ small enough 
exists so that $L(B)$ is not semipositive, then we get a contradiction to $\alpha_1 > \alpha_3$. If not, choose $q = (q_1, q_2, q_3)^t$ such that $Vq = (-1, 1, 1)^t$ 
and proceed with a suitable $B$ as above. Choose $y_2 > 0$ sufficiently small so 
that $L(B)$ is not semipositive. Notice that it is always possible to choose 
$y_2 > 0$ small enough so that $L(B)$ (where $B$ is as above) is not semipositive 
in at least one of the above two cases.
	
\medskip
\noindent
\underline{Case 6:} The case $\alpha_1 < \alpha_3$ can be handled similarly.
	
\noindent
We thus have $\alpha_1 = \alpha_3$ when $\beta_1 = \beta_3, \ \beta_1 \neq 0$.
	
\medskip
\noindent
\underline{Suppose $\beta_1 > \beta_3 \neq 0$.}
	
\medskip
\noindent
\underline{Case 7:} Suppose $\alpha_3 \beta_1 > \alpha_1 \beta_3$. In this case, 
choose a vector $q = (q_1, q_2, q_3)^t$ such that $Vq = (1, 1, -1)^t$. 
Let $B = \begin{bmatrix}
q_1 & 0 & 0\\
q_2 & 0 & 0\\
q_3 & 0 & 0
\end{bmatrix} \ + \ y_2 \begin{bmatrix}
0 & x_1 & 0\\
0 & x_2 & 0\\
0 & x_3 & 0
\end{bmatrix}$. Then, $L(B) = \begin{bmatrix*}[r]
1 & \alpha_1 & -\beta_1\\
1 & \alpha_1 & -\beta_1\\
-1 & -\alpha_3 & \beta_3
\end{bmatrix*} \ + 
\ y_2 L \Bigg (\begin{bmatrix}
0 & x_1 & 0\\
0 & x_2 & 0\\
0 & x_3 & 0
\end{bmatrix} \Bigg)$. Since the first term is not semipositive, it is possible 
to choose a $y_2 > 0$ sufficiently small so that $L(B)$ is not semipositive.
	
\medskip
\noindent
\underline{Case 8:} Suppose $\alpha_3 \beta_1 \leq \alpha_1 \beta_3$. Choose a 
$q = (q_1, q_2,q_3)^t$ such that $Vq = (1, -1, -1)^t$ and consider the matrix $B$ as 
constructed in the previous cases. If there exists a $y_2 > 0$ small enough so that 
$L(B)$ is not semipositive, then we get a contradiction to 
$\alpha_3 \beta_1 \leq \alpha_1 \beta_3$. Else, choose $q$ such that 
$Vq = (-1, 1 -1)^t$ and proceed as above. Once again, note that 
in at least one of these cases, a $y_2 > 0$ that is sufficiently small can be chosen 
so that $L(B)$ is not semipositive.
	
\medskip
\noindent
\underline{Suppose $\beta_3 > \beta_1$.}
	
\medskip
\noindent
\underline{Case 9:} This case can be handled similar to the case when 
$\beta_1 > \beta_3$.
	
\medskip
\noindent
\underline{Suppose $\alpha_1 > \alpha_2 > \alpha_3$ and 
$\beta_1 = \beta_2 \neq \beta_3$.}
	
\medskip
\noindent
\underline{Case 10:} If $\beta_3 > \beta_1$, then choose a vector $q$ so that 
$Vq = (-1, 1, 1)^t$. Construct a matrix semipositive $B$ as before and observe that 
it is possible to choose a $y_2 > 0$ small enough so that $L(B)$ is not semipositive, 
giving a contradiction to $\beta_3 > \beta_1$.
	
\medskip
\noindent
\underline{Case 11:} The case $\beta_1 > \beta_3$ can be handled similarly.
	
\medskip
\noindent
\underline{Case 12:} Suppose $\alpha_1 > \alpha_2 > \alpha_3$ and 
$\beta_1 = \beta_2 = \beta_3 \neq 0$. Choose $q$ so that $Vq = (-1, 1, 1)^t$. Construct 
a semipositive matrix $B$ as before. If there exists a $y_2 > 0$ small enough so that 
$L(B)$ is not semipositive, then we get a contradiction to our assumption. 
Else, choose a vector $q$ so that $Vq = (1, -1, 1)^t$ and proceed as above. Note that 
in at least one of the cases, it is possible to get a semipositive matrix $B$ so that 
$L(B)$ is not semipositive.
	
\medskip
\noindent
\underline{Suppose $\alpha_1 > \alpha_2 > \alpha_3$ and $\beta_1 \neq \beta_2 \neq \beta_3$.}
	
\medskip
\noindent
\underline{Case 13:} Suppose the vectors $(\alpha_1, \alpha_2, \alpha_3)^t$ and 
$(-\beta_1, -\beta_2, -\beta_3)^t$ are linearly dependent. Let $\delta > 0$ be such that 
$(-\beta_1, -\beta_2, -\beta_3)^t = -\delta (\alpha_1, \alpha_2, \alpha_3)^t$. Choose a 
vector $q$ so that $Vq  = (1, -1, 1)^t$ and form a semipositive matrix $B$ as before. 
Observe that $L(B) = \begin{bmatrix*}[r]
1 & \alpha_1 & -\delta \alpha_1\\
-1 & -\alpha_2 & \delta \alpha_2\\
1 & \alpha_3 & -\delta \alpha_3
\end{bmatrix*} \ + \ y_2 L \Bigg( \begin{bmatrix}
0 & x_1 & 0\\
0 & x_2 & 0\\
0 & x_3 & 0
\end{bmatrix} \Bigg)$. 
If there exists a $y_2 > 0$ small enough so that $L(B)$ is not semipositive, 
then we get a contradiction to our assumption. Else, choose a vector $q$ such that 
$Vq = (1, 1, -1)^t$ and proceed as above. Notice that in this case, it is possible 
to choose a $y_2$ small enough that makes $L(B)$ not semipositive in at leat one 
of the cases.
	
\medskip
\noindent
\underline{Case 14:} Suppose the vectors $(\alpha_1, \alpha_2, \alpha_3)^t$ and 
$(-\beta_1, -\beta_2, -\beta_3)^t$ are linearly independent. Consider the invertible 
matrix 
$D = \begin{bmatrix*}[r]
-1 & -\alpha_1 & \beta_1\\
1 & \alpha_2 & -\beta_2\\
-1 & -\alpha_3 & \beta_3
\end{bmatrix*}$. Suppose $\det D < 0$. Notice that 
$D^{-1} = (\frac{1}{\det D}) \begin{bmatrix}
{\ast} & {\ast} & {\ast}\\
{\ast} & {\ast} & {\ast}\\
\alpha_2 - \alpha_3 & \alpha_1 - \alpha_3 & \alpha_1 - \alpha_2
\end{bmatrix}$. Notice that the last rwo entries of the matrix 
are positive. This implies that the vector 
$u = (\frac{1}{\det D}) 
(\alpha_2 - \alpha_3,  \alpha_1 - \alpha_3,  \alpha_1 - \alpha_2)^t < 0$. 
Now $D^t u = - D^t (-u) = (0, 0, 1)^t \geq 0$. Thus, $D$ is not semipositive 
(by the Theorem of the alternative \ref{alternative-orthant}). Now choose a vector 
$q$ so that $Vq = (-1, 1, -1)^t$ and form the semipositive matrix $B$ as in the 
previous steps. 
Then, $L(B) = \begin{bmatrix*}[r]
-1 & -\alpha_1 & \beta_1\\
1 & \alpha_2 & -\beta_2\\
-1 & -\alpha_3 & \beta_3
\end{bmatrix*} \ + \ y_2 L \Bigg( \begin{bmatrix}
0 & x_1 & 0\\
0 & x_2 & 0\\
0 & x_3 & 0
\end{bmatrix} \Bigg)$. Choose
a $y_2$ small enough so that $L(B)$ is not semipositive. If $\det D > 0$, then, 
$(-D)^{-1} = (\frac{1}{- \det D}) \begin{bmatrix}
{\ast} & {\ast} & {\ast}\\
{\ast} & {\ast} & {\ast}\\
\alpha_2 - \alpha_3 & \alpha_1 - \alpha_3 & \alpha_1 - \alpha_2
\end{bmatrix}$. The vector 
$u = (\frac{1}{- \det D}) 
(\alpha_2 - \alpha_3, \alpha_1 - \alpha_3,  \alpha_1 - \alpha_2)^t < 0$ 
and $- (-D^t) (-u) = (0, 0, 1)^t \geq 0$. Once again, by the 
Theorem of the Alternative (Theorem \ref{alternative-orthant}), 
we have $-D \notin S(\reals^3_+)$. Now choose a vector $q$ so that 
$Vq = (1, -1, 1)^t$ and proceed as in the case when $\det D < 0$.
	
\medskip
\noindent
Combining all these steps, we seee that $L(A_1) = \begin{bmatrix}
P_1 & \alpha P_1 & -\beta P_1\\
P_2 & \alpha P_2 & -\beta P_2\\
P_3 & \alpha P_2 & -\beta P_3
\end{bmatrix}, \  \mbox{where} \ 
\alpha > 0$ and $\beta \geq 0$. 
\end{proof}

\medskip
\begin{remark}
	Consider the following flowchart. 
	
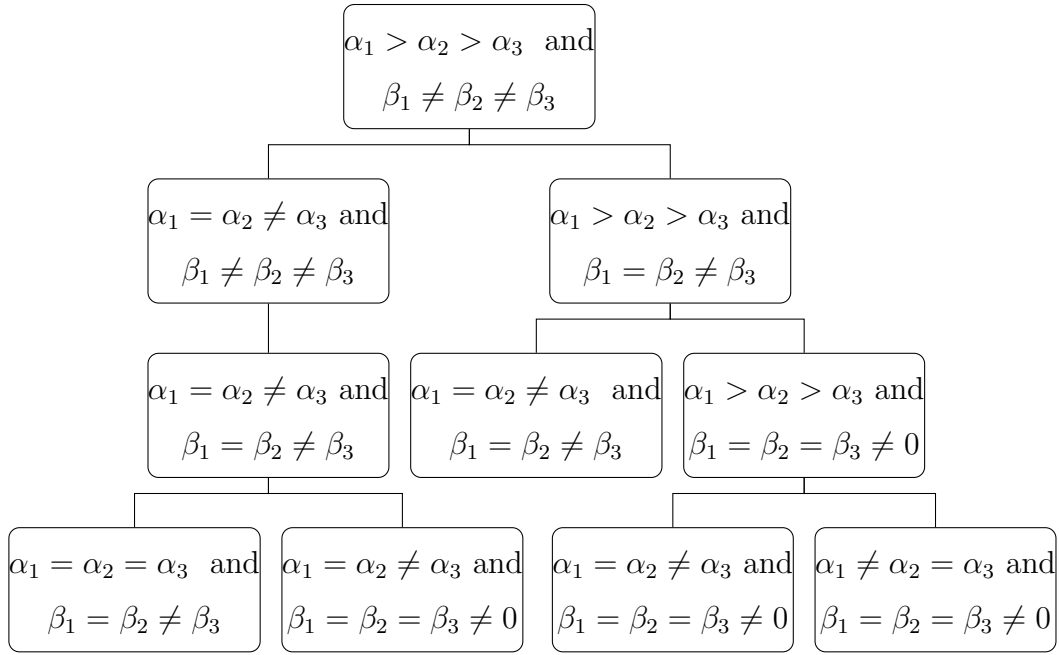
\begin{center}
		\begin{forest}, for tree={draw, rounded corners, minimum width=0.1cm, minimum height=4em},
			forked edges,baseline, qtree
			[$\alpha_1>\alpha_2>\alpha_3$ \ and \\ $\beta_1\neq \beta_2 \neq \beta_3$
			[{$\alpha_1=\alpha_2 \neq \alpha_3$}\ and \\ $\beta_1\neq \beta_2 \neq \beta_3$
			[{$\alpha_1=\alpha_2 \neq \alpha_3$}\ and \\ {$\beta_1 =\beta_2 \neq \beta_3$}
			[{$\alpha_1=\alpha_2 = \alpha_3$} \ and\\ {$\beta_1 =\beta_2 \neq \beta_3$}]
			[{$\alpha_1=\alpha_2 \neq \alpha_3$}\ and\\ {$\beta_1 =\beta_2 = \beta_3 \neq 0$}]
			]]
			[$\alpha_1>\alpha_2>\alpha_3$\ and \\ {$\beta_1 = \beta_2 \neq \beta_3$}
			[{$\alpha_1=\alpha_2 \neq \alpha_3$} \ and\\ {$\beta_1 =\beta_2 \neq \beta_3$}
			]
			[$\alpha_1>\alpha_2>\alpha_3$\ and\\  
			{$\beta_1 = \beta_2 = \beta_3 \neq 0$}
			[{$\alpha_1=\alpha_2 \neq \alpha_3$}\ and\\ {$\beta_1 =\beta_2 = \beta_3 \neq 0$}]
			[{$\alpha_1\neq \alpha_2 = \alpha_3$}\ and\\ {$\beta_1 =\beta_2 = \beta_3 \neq 0$}]
			]
			]]
		\end{forest}
	\end{center}
	
\medskip
\noindent
Lemma \ref{positive-column-1} shows that each of the steps in the above flowchart fails. 
\end{remark}

We now consider the case where the matrix $E$ is not semipositive, but $-E$ is, assuming 
the existence of a positive column.

\begin{lemma}\label{positive-column-2}
Let $L$ be an invertible linear map on $M_3$ that preserves semipositive matrices. 
Suppose the first column of $L(A_1)$ is positive, $-E$ is semipositive, whereas 
$E$ is not. Then, $L(A_1)$ has rank one.
\end{lemma}

\medskip
\begin{proof}
Let $E$ and $V$ be as in Lemma \ref{positive-column-1}. We begin by proving that 
$-E$ does not contain any $2 \times 2$ mininimally semipositive matrix.

\medskip
\noindent
\underline{Claim 1:} $-E$ does not contain any $2 \times 2$ mininimally semipositive 
matrix. 
	
\noindent
Suppose not.
	
\medskip
\noindent
\underline{Case 1:} Suppose $-E = \begin{bmatrix*}[r]
-\alpha_1 & \beta_1\\
\alpha_2 & -\beta_2\\
-\alpha_3 & \beta_3
\end{bmatrix*}$, where $\alpha_1 \geq 0, \alpha_2 > 0, 
\alpha_3 \geq 0, \beta_1 > 0, \beta_2 \geq 0, \beta_3 > 0$. Choose a vector $q$ 
such that 
$Vq = (1, 1, -1)^t$ and form the matrix $B = \begin{bmatrix}
q_1 & 0 & 0\\
q_2 & 0 & 0\\
q_3 & 0 & 0
\end{bmatrix} \ + \ y_2 \begin{bmatrix}
0 & x_1 & 0\\
0 & x_2 & 0\\
0 & x_3 & 0
\end{bmatrix}, \ y_2 > 0$, 
so that $B$ is semipositive. Then, $L(B) = \begin{bmatrix*}[r]
1 & \alpha_1 & -\beta_1\\
1 & -\alpha_2 & \beta_2\\
-1 & -\alpha_3 & \beta_3
\end{bmatrix*} \ + \ y_2 L \Bigg(\begin{bmatrix}
0 & x_1 & 0\\
0 & x_2 & 0\\
0 & x_3 & 0
\end{bmatrix} \Bigg)$. 
If there exists a $y_2> 0$ small enough so that $L(B)$ is not semipositive, then 
we get a contradiction. Else, choose a vector $q$ such that $Vq = (-1, 1,1)^t$ and 
proceed as above. In at least one of these cases, a $y_2$ small enough will exist 
such that $L(B)$ is not semipositive.
	
\medskip
\noindent
\underline{Case 2:} Suppose $-E = \begin{bmatrix*}[r]
-\alpha_1 & \beta_1\\
\alpha_2 & -\beta_2\\
\alpha_3 & -\beta_3
\end{bmatrix*}$, where $\alpha_1 \geq 0, \alpha_2 > 0, 
\alpha_3 > 0, \beta_1 > 0, \beta_2 \geq 0, \beta_3 \geq 0$. 
This case can be handled similar to Case 1.
	
\medskip
\noindent
\underline{Case 3:} Suppose $-E = \begin{bmatrix*}[r]
-\alpha_1 & \beta_1\\
\alpha_2 & -\beta_2\\
\alpha_3 & \beta_3
\end{bmatrix*}$, where $\alpha_1 \geq 0, \alpha_2 > 0, 
\alpha_3 > 0, \beta_1 > 0, \beta_2 \geq 0, \beta_3 > 0$. Choose a vector $q$ such that 
$Vq = (-1, -1, 1)^t$ and construct a semipositive matrix $B$ as before. Then
$L(B) = \begin{bmatrix*}[r]
-1 & -\alpha_1 & \beta_1\\
-1 & \alpha_2 & -\beta_3\\
1 & -\alpha_3 & -\beta_3
\end{bmatrix*} \ + \ y_2 L \Bigg(\begin{bmatrix}
0 & x_1 & 0\\
0 & x_2 & 0\\
0 & x_3 & 0
\end{bmatrix} \Bigg)$. Choosing $y_2$ small enough, 
we can make $L(B)$ not semipositive. 

\medskip	
\noindent
Cases 1, 2 and 3 prove Claim 1. 

\medskip	
\noindent
Since $E$ is not semipositive, but $-E$ is, we see that $E = \begin{bmatrix*}[r]
-\alpha_1 & -\beta_1\\
-\alpha_2 & \beta_2\\
-\alpha_3 & \beta_3
\end{bmatrix*}$, where 
$\alpha_i > 0$ for $i = 1, 2, 3$, $\beta_1 \geq 0, \beta_2, \beta_3 \in \reals$.
	
\medskip
\noindent
Suppose $\beta_2 > 0$ and $\beta_3 > 0$. Choose a vector $q$ such that 
$Vq = (1, -1, 1)^t$ and construct a semipositive matrix $B$ as before. Then, 
$L(B) = \begin{bmatrix*}[r]
1 & -\alpha_1 & -\beta_1\\
-1 & \alpha_2 & \beta_2\\
1 & -\alpha_3 & \beta_3
\end{bmatrix*} \ + \ y_2 L \Bigg(\begin{bmatrix}
0 & x_1 & 0\\
0 & x_2 & 0\\
0 & x_3 & 0
\end{bmatrix} \Bigg)$. If there exists a $y_2$ small enough such that 
$L(B)$ is not semipositive, then we get a contradiction to our assumption. If not, 
choose a vector $q$ so that $Vq = (1, 1, -1)^t$ and proceed as above. As before, 
note that it is possible to choose a suitable $y_2 > 0$ that makes $L(B)$ not 
semipositive in at least one of the cases. Thus, $E = \begin{bmatrix*}[r]
-\alpha_1 & -\beta_1\\
-\alpha_2 & -\beta_2\\
-\alpha_3 & \beta_3
\end{bmatrix*}$, where $\alpha_i > 0$ for $i = 1, 2, 3$, 
$\beta_1 \geq 0, \beta_2 \geq 0, \beta_3 \in \reals$. If $\beta_3 > 0$, then choose a 
$q$ such that $Vq = (1, -1, -1)^t$ or $(-1, 1, 1)^t$ and proceed as above. Therefore, 
$E = \begin{bmatrix}
-\alpha_1 & -\beta_1\\
-\alpha_2 & -\beta_2\\
-\alpha_3 & -\beta_3
\end{bmatrix}$, where $\alpha_i > 0$ and $\beta_i \geq 0$ for $i = 1, 2, 3$. We can now 
proceed as in \ref{positive-column-1} to show that all the $\alpha_i$'s are equal, 
all the $\beta_i$'s are equal. This finishes the proof.
\end{proof}

The next case is the following.

\begin{lemma}\label{positive-column-3}
Let $L$ be an invertible linear map on $M_3$ that preserves semipositive matrices. 
Suppose the first column of $L(A_1)$ is positive and both $E$ as well as $-E$ are not semipositive. Then, $L(A_1)$ has rank one.
\end{lemma}

\begin{proof}
Assuming that the second column of $L(A_2)$ is positive, write 
$L(A_2) = \begin{bmatrix}
\mu_1 Q_1 & Q_1 & \delta_1 Q_1\\
\mu_2 Q_2 & Q_2 & \delta_2 Q_2\\
\mu_3 Q_3 & Q_3 & \delta_3 Q_3
\end{bmatrix}$, where 
$Q_1 = l_{22}x_1 + l_{25}x_2 + l_{28}x_3, Q_2 = l_{52}x_1 + l_{55}x_2 + l_{58}x_3, 
Q_3 = l_{82}x_1 + l_{85}x_2 + l_{88}x_3, \mu_i, \delta_i \in \reals, 
i = 1, 2, 3$. Choose a vector $p = (p_1, p_2, p_3)^t \in \reals^3$ such that 
$L\Bigg (\begin{bmatrix}
          p_1 & 0 & 0\\
          p_2 &0 & 0\\
          p_3 & 0 & 0
         \end{bmatrix} \Bigg) = \begin{bmatrix}
         						  1 & \alpha_1 & \beta_1\\
         						  1 & \alpha_2 & \beta_2\\
         						  1 & \alpha_3 & \beta_3
								\end{bmatrix}$. If $A_2$ is semipositive, the the 
matrix $\begin{bmatrix}
-p_1 & 0 & 0\\
-p_2 & 0 & 0\\
-p_3 & 0 & 0
\end{bmatrix} + A_2$ is semipositive and so is its image under $L$. It is possible 
to choose a $y_2 > 0$ so that the first column of the image of the above matrix under 
$L$ is negative. Now, choose a $q = (q_1, q_2, q_3)^t \in \reals^3$ such that 
$L \Bigg( \begin{bmatrix}
          p_1 & 0 & 0\\
          p_2 &0 & 0\\
          p_3 & 0 & 0
         \end{bmatrix} - y_2 \begin{bmatrix}
          					  0 & q_1 & 0\\
          				      0 & q_2 & 0\\
          					  0 & q_3 & 0
         					 \end{bmatrix} \Bigg) = 
\begin{bmatrix}
1 - \mu_1 y_2 & \alpha_1 - y_2 & \beta_1 - \delta_1 y_2\\
1 - \mu_2 y_2 & \alpha_2 - y_2 & \beta_2 - \delta_2 y_2\\
1 - \mu_3 y_2 & \alpha_3 - y_2 & \beta_3 - \delta_3 y_2
\end{bmatrix}$. Choose $y_2 > 0$ sufficiently small so that the first column of the 
above matrix is positive and negative of the submatrix formed from the last two 
columns of the above matrix is semipositive. It then follows from the previous case 
that $\alpha_1 = \alpha_2 = \alpha_3 = 0$. Similarly, we can show that 
$\beta_1 = \beta_2 = \beta_3 = 0$. These two put together implies that $E = 0$. 
Thus, $L(A_1) = \begin{bmatrix}
P_1 \\
P_2 \\
P_3
\end{bmatrix} \begin{bmatrix}
1 & 0 & 0
\end{bmatrix}$ and therefore has rank one in this case too.  
\end{proof} 

A similar argument shows that $L(A_2)$ and $L(A_3)$ have rank one. We now have the 
result on the structure of a preserver in this case.

\begin{theorem}\label{3x3-main-theorem}
Let $L$ be an invertible linear map on $M_3$ that preserves $S(\reals_+^3)$. Then 
$L(A) = XAY$ for all $A \in M_{3}$, for some invertible row positive matrix $X$ and an 
inverse nonnegative matrix $Y$.
\end{theorem}

\begin{proof}
We know that $L(A_i)$ is a rank one semipositive matrix for $i = 1, 2, 3$. We also have 
$L(A_1) = \begin{bmatrix}
P_1\\ P_2\\ P_3
\end{bmatrix} \begin{bmatrix}
1 & \alpha & \beta
\end{bmatrix}$ for some $\alpha, \beta \in \reals$. Without loss 
of generality, assume that the second column of $L(A_2)$ is positive and write $L(A_2)$ 
as $\begin{bmatrix}
Q_1\\ Q_2\\ Q_3
\end{bmatrix} \begin{bmatrix}
\nu & 1 & \theta
\end{bmatrix}$ for some $\nu, \theta \in \reals$.
	
\medskip
\noindent
\underline{Claim 1:} The vectors $(l_{11}, l_{14}, l_{17})^t$ and $(l_{22}, l_{25}, l_{28})^t$ are linearly dependent. If not, then there exists $(u_1, u_2, u_3)^t$ 
and $(v_1, v_2, v_3)^t$ 
such that $\begin{bmatrix}
l_{22} & l_{25} & l_{28}\\
l_{11} & l_{14} & l_{17}
\end{bmatrix} \begin{bmatrix}
u_1 & v_1\\
u_2 & v_2\\
u_3 & v_3
\end{bmatrix} = \begin{bmatrix}
1 & 0\\
0 & 1
\end{bmatrix}$. 
Since $\begin{bmatrix}
l_{22} & l_{25} & l_{28}\\
l_{11} & l_{14} & l_{17}
\end{bmatrix}$ is a nonnegative matrix of rank two, the matrix 
$B = \begin{bmatrix}
u_1 & v_1 & 0\\
u_2 & v_2 & 0\\
u_3 & v_3 & 0
\end{bmatrix}$ is a semipositive matrix. Then, $L(B) = \begin{bmatrix}
0 & 0 & 0\\
{\ast} & {\ast} & {\ast}\\
{\ast} & {\ast} & {\ast}
\end{bmatrix}$, which is not 
semipositive. Thus, $(l_{11}, l_{14}, l_{17})^t$ and $(l_{22}, l_{25}, l_{28})^t$ are 
dependent vectors and so 
$(l_{11}, l_{14}, l_{17})^t = \lambda_1 (l_{22}, l_{25}, l_{28})^t$ 
for some $\lambda_1 > 0$. Similarly, it can be shown that 
$(l_{52}, l_{55}, l_{58})^t = \lambda_2 (l_{41}, l_{44}, l_{47})^t$ and 
$(l_{82}, l_{85}, l_{88})^t = \lambda_3 (l_{71}, l_{74}, l_{77})^t$, where 
$\lambda_2$ and $\lambda_3$ are positive real numbers.
	
\noindent
Now, assume that the last column of $L(A_3)$ is positive. We then have  
$L(A_3) = \begin{bmatrix}
R_1\\ R_2\\ R_3
\end{bmatrix} \begin{bmatrix}
\mu & \delta & 1
\end{bmatrix}$ for some $\mu, \delta \in \reals$. It can be seen that 
$(l_{33}, l_{36}, l_{39})^t = \phi_1 (l_{11}, l_{14}, l_{17})^t$, 
$(l_{63}, l_{66}, l_{69})^t = \phi_2 (l_{41}, l_{44}, l_{47})^t$ and 
$(l_{93}, l_{96}, l_{99})^t = \phi_3 (l_{71}, l_{74}, l_{77})^t$, for some 
positive real numbers $\phi_1, \phi_2$ and $\phi_3$.
	
\medskip
\noindent
We thus have $L(A_1) = \begin{bmatrix}
P_1\\ P_2\\ P_3
\end{bmatrix} \begin{bmatrix}
1 & \alpha & \beta
\end{bmatrix}, \ L(A_2) = \begin{bmatrix}
\lambda_1 P_1\\ \lambda_2 P_2\\ \lambda_3 P_3
\end{bmatrix} \begin{bmatrix}
\nu & 1 & \theta
\end{bmatrix} \ \mbox{and} \ L(A_3) = \begin{bmatrix}
\phi_1 P_1\\ \phi_2 P_2\\ \phi_3 P_3
\end{bmatrix} \begin{bmatrix}
\mu & \delta & 1
\end{bmatrix}$, where $\alpha, \beta, \nu, \theta, \mu$ and $\delta$ are real numbers.
	
\medskip
\noindent
\underline{Claim 2:} $\lambda_1 = \lambda_2 = \lambda_3$ and $\phi_1 = \phi_2 = \phi_3$. 
Suppose $\lambda_1 > \lambda_2$. Choose positive numbers $d_1, d_2$ and $d_3$. Define 
$B = \begin{bmatrix}
l_{11} & l_{14} & l_{17}\\
l_{41} & l_{44} & l_{47}\\
l_{71} & l_{74} & l_{77}
\end{bmatrix}^{-1}$ \ $\begin{bmatrix*}[c]
d_1 & \frac{-d_1}{\lambda_1}  & 0\\
-d_2 & \frac{d_2}{\lambda_2} & 0\\
0 & 0 & \frac{d_3}{\phi_3}
\end{bmatrix*}$. Since $B$ is inverse nonnegative, it is semipositive. 
We have \\ $L(B) = \begin{bmatrix}
d_1\\ -d_2\\ 0
\end{bmatrix}\begin{bmatrix}
1 & \alpha & \beta
\end{bmatrix} + \begin{bmatrix}
-d_1\\ d_2\\ 0
\end{bmatrix} 
\begin{bmatrix}
\nu & 1 & \theta
\end{bmatrix} + \begin{bmatrix}
0\\ 0\\ d_3
\end{bmatrix} 
\begin{bmatrix}
\mu & \delta & 1
\end{bmatrix}\\= 
\begin{bmatrix}
(1 - \nu) d_1 & -(1 - \alpha) d_1 & (\beta - \theta)d_1\\
-(1-\nu) d_2 & (1-\alpha) d_2 & - (\beta - \theta) d_2\\
\mu d_3 & \delta d_3 & d_3
\end{bmatrix}$, which is not semipositive (notice that each $2 \times 2$ 
submatrix of the matrix formed from the first two rows of $L(B)$ is not semipositive). 
Thus, $\lambda_1 \leq \lambda_2$. 
	
\noindent
Similarly, one can show that $\lambda_1 \geq \lambda_2$, thereby proving that they are equal. Proceeding similarly, we can show that $\lambda_1 = \lambda_2 = \lambda_3$ and 
$\phi_1 = \phi_2 = \phi_3$. Combining all of these, we see that 
$L(A) = XAY$, where $A = \begin{bmatrix}
a_{11} & a_{12} & a_{13}\\
a_{21} & a_{22} & a_{23}\\
a_{31} & a_{32} & a_{33}
\end{bmatrix}, \ X = \begin{bmatrix}
	l_{11} & l_{14} & l_{17}\\
	l_{41} & l_{44} & l_{47}\\
	l_{71} & l_{74} & l_{77} 
\end{bmatrix} \ \mbox{and} \ 
Y = \begin{bmatrix}
	1 & \alpha & \beta\\
	\lambda_1 \nu & \lambda_1 & \lambda_1 \theta\\
	\mu \phi_1 & \delta \phi_1 & \phi_1
\end{bmatrix}$. It is clear that $X$ is row positive. Inverse nonnegativity of $Y$ 
follows as the map preserves semipositivity.
\end{proof}

\end{document}